\documentclass[11pt]{article}

\usepackage[margin=1in]{geometry}
\usepackage{amsmath, amssymb, amsthm}
\usepackage{booktabs}
\usepackage{array}
\usepackage{enumitem}
\usepackage{graphicx}
\usepackage{hyperref}
\usepackage{listings}
\usepackage{tikz}

\newtheorem{theorem}{Theorem}
\newtheorem{proposition}[theorem]{Proposition}
\newtheorem{lemma}[theorem]{Lemma}
\newtheorem{corollary}[theorem]{Corollary}
\newtheorem{conjecture}[theorem]{Conjecture}

\theoremstyle{definition}
\newtheorem{remark}[theorem]{Remark}
\newtheorem{example}[theorem]{Example}
\newtheorem{definition}[theorem]{Definition}

\newcommand{\G}{\mathcal{G}}
\newcommand{\Z}{\mathbb{Z}}

\title{Holes in Valid-Extension Sets of Finite Gilbreath Sequences}
\author{Leila Muney}
\date{}

\begin{document}

\maketitle

\begin{abstract}
Given a finite sequence of integers, form its difference triangle by
repeatedly taking absolute differences of consecutive entries. We call
the sequence \emph{Gilbreath} if the leftmost entry of every row below
the top is $1$. The Gilbreath conjecture, which remains open, asserts that every initial segment of the primes is a Gilbreath sequence.

This paper studies the local extension problem: given a Gilbreath
sequence, which integers can be appended to it while preserving the
Gilbreath property? We call the set of such admissible values the
\emph{valid-extension set} of the sequence. A previously proposed
characterization in the literature predicts that this set always fills
a natural parity interval around the last term. We show that this
fails in general: the valid-extension set can have interior holes,
with the smallest failure occurring at length~$5$ for the sequence
$(2,3,5,9,15)$.

The paper develops a corrected extension set theory. We give an exact criterion
for membership in the valid-extension set, an algorithm that computes
it, and a sharp condition determining exactly when the set fills the
candidate interval. This last condition is an order-sensitive analogue
of the classical Brown completeness criterion for subset sums. We also
establish endpoint validity and reflection symmetry, determine the
exact minimum size of the valid-extension set together with its unique
minimizer, exhibit a family whose valid-extension set has exponentially
many components, and provide enumeration data through length~$11$.
\end{abstract}

\section{Introduction}\label{sec:intro}

Given a finite integer sequence $S = (s_1, \ldots, s_n)$, its
\emph{difference triangle} is defined by $s_a^0 := s_a$ and
\[
s_a^b := \left| s_{a+1}^{b-1} - s_a^{b-1} \right|
\quad \text{for } b \geq 1,\ 1 \leq a \leq n - b.
\]
The sequence is \emph{Gilbreath} if $s_1^b = 1$ for every
$1 \leq b \leq n - 1$. Figure~\ref{fig:generic-difference-triangle}
depicts the construction of the triangle for $n=5$.

\begin{figure}[htbp]
\centering
\[
\begin{array}{ccccc}
s_1 & s_2 & s_3 & s_4 & s_5 \\[0.6em]
& s_1^1 & s_2^1 & s_3^1 & s_4^1 \\[0.6em]
& & s_1^2 & s_2^2 & s_3^2 \\[0.6em]
& & & s_1^3 & s_2^3 \\[0.6em]
& & & & s_1^4
\end{array}
\]
\caption{Difference triangle of an arbitrary sequence of length $n=5$.
The sequence is Gilbreath if and only if
$s_1^1=s_1^2=s_1^3=s_1^4=1$.}
\label{fig:generic-difference-triangle}
\end{figure}

The iterated absolute-difference triangle was first studied by Proth
and rediscovered by Gilbreath in the context of the prime sequence.
Computer experiments by Killgrove and Ralston~\cite{killgrove} and
most extensively by Odlyzko~\cite{odlyzko}, the latter for primes up
to $10^{13}$ (about $3.4 \times 10^{11}$ primes), provide substantial
numerical evidence that initial segments of the primes are Gilbreath.
The assertion that this property holds for every initial segment of
the primes is known as \emph{Gilbreath's conjecture}, and remains
open. It is recorded as Problem~A10 in Guy~\cite{guy} and as Appendix
Problem~68 in Montgomery~\cite{montgomery}.

The features of the prime sequence responsible for the Gilbreath
property have been investigated through several lenses. Croft,
reported by Gardner~\cite{gardner}, observed that the property does
not seem to depend on primality in any deep sense: he conjectured
that any sequence beginning with $2$, continuing with odd numbers,
and having sufficiently small gaps should be Gilbreath. This
small-gap heuristic was later formalized in probabilistic form by
Chase~\cite{chase}, who proved that sequences beginning $2, 3$ with
random small gaps are almost surely Gilbreath.

Our finite family $\G_n$ (defined below) is naturally aligned with
this small-gap perspective, but the results in this paper do not
address the Gilbreath conjecture directly. Instead, we focus on a
finite, local question: \emph{given a Gilbreath sequence, which
integers can be appended to it while preserving the Gilbreath
property?}

For a Gilbreath sequence $S$, the \emph{valid extension set} is
\[
K_S := \{k \in \Z : (s_1, \ldots, s_n, k) \text{ is Gilbreath}\}.
\]
We call the cardinality $|K_S|$ the \emph{extension width} of $S$. We
work with the family
\[
\G_n := \{S = (s_1, \ldots, s_n) :
S \text{ is strictly increasing, Gilbreath, and } (s_1, s_2) = (2, 3)\}.
\]
The choice $(s_1, s_2) = (2, 3)$ is arbitrary: a shift argument
(Section~\ref{sec:shift}) shows that $|K_S|$ and the structure of
$K_S$ depend only on the gap sequence, so all results extend to the
shifted family with initial pair $(a, a+1)$ for any integer $a$.

\subsection{Note on a previously claimed interval characterization}\label{sec:gatti-note}

Gatti~\cite{gatti2023} introduces a nested-absolute-value equation
\[
\left| s_1^{n-1} - \left| s_2^{n-2} - \left| \cdots - |s_n - k| \cdots
\right| \right| \right| = 1
\]
characterizing membership of $k$ in $K_S$, and proposes to unfold this
into an independently signed sum
\[
k = \pm s_1^{n-1} \pm s_2^{n-2} \pm \cdots \pm s_{n-1}^1 + s_n \pm 1,
\]
treating the $n$ signs as freely chosen in $\{+1, -1\}$. However, the
signs in this unfolding are not independent in general; some
independent sign choices produce values that do not return a
Gilbreath sequence when appended to $S$.
Gatti~\cite{gatti2023} further claims that the set of all possible
values attainable from this formula produces a parity interval that
is equal to $K_S$. We show in Section~\ref{sec:signed-sum} that this
is not the case: the signed formula can both miss values in the candidate interval and include values that are not valid extensions.

The present paper develops a corrected extension-set theory. We give an exact
algorithm for $K_S$, clarify the relationship between the signed-sum
set of~\cite{gatti2023}, the candidate interval $C_S$, and the true
valid-extension set $K_S$, identify the surviving features (endpoint
validity, parity, and reflection symmetry), characterize precisely
when $K_S = C_S$, and study the extremal and disconnectedness
behavior of $K_S$.

\subsection{Notation guide}

For ease of reference, we collect the main notation used throughout the paper below.

\begin{description}[leftmargin=2.8cm, style=nextline]

\item[\(\G_n\)]
Strictly increasing Gilbreath sequences of length \(n\) beginning with
\((2,3)\).

\item[\(K_S\)]
The valid-extension set: all \(k\in\Z\) such that \((S,k)\) is Gilbreath.

\item[\(K_S^+\)]
The increasing valid extensions:
\[
K_S^+=\{k\in K_S:k>s_n\}.
\]

\item[\(e_i\)]
The right anti-diagonal entry
\[
e_i=s_{n-i}^i.
\]

\item[\(A(S)\)]
The anti-diagonal sum:
\[
A(S)=\sum_{i=1}^{n-1}e_i.
\]

\item[\(r_i\)]
The new right-edge entries created after appending a proposed extension
\(k\):
\[
r_0=|k-s_n|,\qquad r_i=|r_{i-1}-e_i|.
\]

\item[\(F_S\)]
The folding map determined by the right anti-diagonal:
\[
F_S(d)=\left|\cdots\left||d-e_1|-e_2\right|\cdots-e_{n-1}\right|.
\]
Thus \(F_S(|k-s_n|)=r_{n-1}\).

\item[\(C_S\)]
The candidate set:
\[
C_S = \{k \in \Z : |k - s_n| \leq A(S)+1,\ k \equiv s_n \pmod 2\};
\]
the parity-compatible interval of radius \(A(S)+1\) around \(s_n\).

\item[\(H_S\)]
The hole set:
\[
H_S=C_S\setminus K_S.
\]

\item[\(h(S)\)]
The defect:
\[
h(S)=|H_S|=|C_S|-|K_S|.
\]

\item[\(S_{\pm}\)]
The signed-sum set obtained by treating the signs in the unfolded
absolute-value expression as independent.

\item[\(W_S\)]
The weight multiset \(W_S=\{e_1,\ldots,e_{n-1},1\}\) associated to the
signed-sum relaxation.

\item[\(\Sigma(W)\)]
The set of subset sums of a multiset \(W\).

\item[\(D_S\)]
The valid distance set:
\[
D_S=\{|k-s_n|:k\in K_S\}.
\]

\item[\(P_e\)]
The reverse preimage step for \(x\mapsto |x-e|\).

\item[\(T_i\)]
The unnormalized reverse-tree sets used to compute \(D_S\).

\item[\(Q_a\)]
The normalized preimage step after dividing by \(2\).

\item[\(\widetilde{T}_i,L_i,a_i\)]
The normalized reverse-tree sets, interval lengths, and normalized
anti-diagonal entries used in the interval-completeness criterion.

\item[\(L_n\)]
The minimal sequence
\[
(2,3,5,7,\ldots,2n-1).
\]

\item[\(U_n\)]
The doubling sequence
\[
(2,3,5,9,17,\ldots,2^{n-1}+1).
\]

\item[\(V_n\)]
The component-doubling family from Section~\ref{sec:Vn}.

\item[\(M_n\)]
The maximum extension width:
\[
M_n=\max_{S\in\G_n}|K_S|.
\]

\item[\(m_n\)]
The minimum extension width:
\[
m_n=\min_{S\in\G_n}|K_S|.
\]

\item[\(N_n\)]
The number of sequences in \(\G_n\):
\[
N_n=|\G_n|.
\]

\end{description}

\subsection{Summary of main results}\label{sec:summary}

The paper has three main parts. We summarize them here and indicate
where the main results are proved.

First, we give an exact criterion for valid extensions. The right
anti-diagonal of the difference triangle determines an iterated
absolute-value map \(F_S\), and a proposed extension \(k\) is valid
exactly when
\[
F_S(|k-s_n|)=1
\]
(Proposition~\ref{prop:criterion}). This identifies the valid distance
set as a fiber of a composition of folding maps and leads to the
reverse-tree algorithm for computing \(K_S\) exactly
(Proposition~\ref{prop:reverse}). The criterion immediately yields the candidate bound
$K_S\subseteq C_S$ (Corollary~\ref{cor:candidate}), and a short
parity argument gives endpoint validity and reflection symmetry of
$K_S$ (Theorems~\ref{thm:endpoints} and \ref{thm:symmetry}).

Second, we compare the true extension set with the signed-sum relaxation
implicit in~\cite{gatti2023}. We show that the signed-sum set is an
affine image of a subset-sum set associated to the right anti-diagonal
(Theorem~\ref{thm:subset-reformulation}). Thus the question of when the
signed sums fill the natural candidate interval \(C_S\) is governed by
Brown's classical criterion for when subset sums fill a full interval.
The equality \(K_S=C_S\) is more rigid: the signs must be compatible
with the ordered nested absolute-value recurrence. Our main structural
theorem gives the exact ordered analogue:
\[
K_S=C_S
\quad\Longleftrightarrow\quad
e_i\leq 1+\sum_{j>i}e_j
\quad (1\leq i\leq n-2).
\]
This is Theorem~\ref{thm:criterion}.

Third, we study the consequences of this criterion. We identify the
first failure of interval-completeness: for \(n\leq 4\) all sequences in
\(\G_n\) are interval-complete, while at \(n=5\) the unique
counterexample is \((2,3,5,9,15)\), with a single hole at \(15\)
(Theorem~\ref{thm:hole}). We determine the minimum possible extension
width, showing that it is \(5\) for every \(n\geq 3\), uniquely achieved
by \(L_n=(2,3,5,7,\ldots,2n-1)\)
(Theorem~\ref{thm:min}). We also compute the extension width of the doubling sequence
$U_n=(2,3,5,9,17,\ldots,2^{n-1}+1)$, obtaining $|K_{U_n}|=2^{n-1}+1$
(Theorem~\ref{thm:doubling}); exhaustive computation through $n\leq 10$
shows this value is the maximum extension width in $\G_n$, giving rise
to Conjecture~\ref{conj:max}. We construct an explicit family
$V_n\in\G_n$ whose valid-extension set has exactly $2^{n-4}$ connected
components in the parity lattice (Theorem~\ref{thm:Vn}), so the
maximum component count over $\G_n$ grows exponentially in $n$.
Finally, we give enumeration data for $N_n=|\G_n|$ through $n\leq 11$
and extremal data through $n\leq 10$ (Section~\ref{sec:data}).

\subsection{Related work}\label{sec:related}

Beyond the historical references in the introduction, the present
paper relates to several active threads in additive number theory and
combinatorial dynamics.

The signed-sum relaxation arising in this paper connects the extension problem to the classical theory of subset sums and complete sequences. Brown's
criterion gives a necessary and sufficient condition for the subset
sums of a finite sequence of nonnegative integers to fill the entire
interval from $0$ to their total sum~\cite{brown1961}. This is the
finite completeness criterion used in Section~\ref{sec:subset-sum} to
characterize when $S_\pm = C_S$. Complete sequences and related
subset-sum questions also appear in the Erd\H{o}s line of additive
number theory: Burr and Erd\H{o}s studied Ramsey-type completeness
properties~\cite{burr-erdos}, and Conlon, Fox, and Pham recently
resolved several problems on subset sums, completeness, and
colorings, including questions of Burr and
Erd\H{o}s~\cite{conlon-fox-pham}. In this terminology, the classical
completeness criterion governs when the signed sums fill the
candidate interval. Our main interval-completeness theorem identifies
a strictly stronger order-sensitive condition governing when the
true valid-extension set fills the interval. The gap between these
two conditions captures the consistency required by the nested
absolute-value structure.

Chase~\cite{chase}, mentioned in the introduction, formalizes
Croft's small-gap heuristic in probabilistic form. Like ours,
Chase's model fixes the starting point $2, 3$ and treats subsequent
entries as free, but the questions are different: Chase studies
probabilistic eventual Gilbreath behavior, whereas we study the local
finite extension problem of determining exactly which next values
preserve the Gilbreath property. More recently, Chase, Hunter, and
Tao~\cite{chase-hunter-tao} combine probabilistic and deterministic
approaches to Gilbreath's conjecture, proving a Cram'er random-model
analogue and establishing a deterministic inverse theorem that
identifies the principal obstructions to the Gilbreath property under
suitable assumptions on prime gaps. Their work seeks to understand
global mechanisms governing Gilbreath's conjecture, whereas our focus
is complementary: we develop an exact structural theory of the finite
valid-extension set $K_S$, giving complete characterizations,
algorithms, and extremal results for the local extension problem.

Granville~\cite{granville2026} also studies Gilbreath's conjecture
from a global perspective, developing a framework based on sieving,
reverse sieving, and equivalence classes of finite sequences to reduce
the conjecture to a collection of representative cases. While his
approach likewise aims to understand the conjecture itself, our work
instead analyzes the finite extension problem for a fixed Gilbreath
sequence, characterizing the exact set of admissible next values and
the combinatorial structure of the resulting valid-extension set
$K_S$.

Bhat, Cobeli, and Zaharescu~\cite{bcz2024} study the same
Proth--Gilbreath triangle as a discrete dynamical system, introducing
the operator $\Upsilon$ that sends the top row to the left edge and
analyzing its six-fold ``helicoidal'' iteration, an associated
$\mathbb{F}_2[[X]]$ involution
$T(f)(X) = f\!\big(X/(1+X)\big) \cdot (1+X)^{-1}$, and the
statistical distribution of $0$'s and nonzero entries along rays
parallel to an edge. Their padding construction
(\cite[Prop.~3.1]{bcz2024}) builds a triangle backwards from a
prescribed southern vertex by choosing eastern-edge values. This is
reminiscent of our reverse-tree process (Section~\ref{sec:reverse}),
but the inverse problem is different: their construction pads the
\emph{eastern} edge to realize a single target apex, whereas our
reverse tree runs up the \emph{right anti-diagonal} from the apex
value $1$ to enumerate the entire valid-extension set $K_S$.
Equivalently, our valid distance set is the fiber over $1$ of an
ordered composition of folding maps $x \mapsto |x - e_i|$, with the
fold parameters supplied by the right anti-diagonal. Earlier work in
this dynamical direction includes the Proth--Gilbreath analogue of
Caragiu, Zaharescu, and Zaki~\cite{czz} and the quasi-periodicity
study of Bhat, Cobeli, and Zaharescu~\cite{bcz-quasi}.

Agama~\cite{agama} reformulates Gilbreath's conjecture through a
``gap sequence / path / circuit'' framework. While both Agama's
paper and ours provide a finite structural reframing, our machinery
and goals are different. Agama examines gap sequences through path
combinatorics, while we examine the combinatorial and additive
structure of the extension set of a fixed finite sequence.

It is important to note that the counts $N_n = |\G_n|$ coincide
(after an index shift) with OEIS sequence~\cite{A080839}, where a
comment of T.~D.~Noe already identifies that the slowest- and
fastest-growing length-$n$ Gilbreath sequences are the minimal
sequence $L_n = (2, 3, 5, 7, \ldots, 2n - 1)$ and the doubling
sequence $(2, 3, 5, 9, 17, \ldots)$, respectively. Our minimum
extension-width theorem (Theorem~\ref{thm:min}) and doubling-sequence
extension-width formula (Theorem~\ref{thm:doubling}) show that these
sequences are also extremal for extension width: the minimal sequence
is the unique minimizer for all $n \geq 3$, and exhaustive
computation shows that the doubling sequence is the maximizer for
$n \leq 10$. We attribute the growth extremizer identification
to~\cite{A080839} and claim novelty only for the cardinality formulas
of the corresponding $K_S$ and the general structural theory.

For clarity, we explicitly state what we do and do not claim as new.
We do \emph{not} claim novelty for the enumeration $N_n = |\G_n|$,
which coincides with OEIS~\cite{A080839}, nor for the identification
of the minimal sequence and doubling sequence as the extremal-growth
sequences, which is stated as a note there. We also do not claim
novelty for the classical subset-sum completeness criterion used to
analyze $S_\pm$, which is due to Brown~\cite{brown1961}. The
contributions we believe to be new are the structural theory of the
valid-extension set $K_S$: the exact membership criterion, the
reverse-tree algorithm producing $K_S$, the interpretation of
$K_S = C_S$ as an ordered folding analogue of classical subset-sum
completeness, and the interval-completeness criterion
(Theorem~\ref{thm:criterion}). We also identify the first hole, prove
the exact minimum extension-width theorem with uniqueness
(Theorem~\ref{thm:min}), and identify the structural properties of
the exponentially disconnected family $V_n$ (Theorem~\ref{thm:Vn}).
The correction to the interval-filling claim of~\cite{gatti2023} is
the conceptual starting point, but the paper develops a broader
extension-set theory around it. We note that these originality claims
are based on searches in the literature and databases, and we would
welcome correction.

For more context on iterated-difference and difference-triangle
sequences, see the OEIS discussion in Section~\ref{sec:oeis}.

\section{The exact extension criterion}\label{sec:criterion-section}

Throughout, for \(S\in\G_n\) we define the \emph{right anti-diagonal}
\[
e_i:=s_{n-i}^i,\qquad 1\leq i\leq n-1.
\]
Thus \(e_1=s_{n-1}^1=s_n-s_{n-1}\) is the last gap, while
\(e_{n-1}=s_1^{n-1}=1\) is the bottom entry of the triangle. We also set
\[
A(S):=\sum_{i=1}^{n-1} e_i.
\]

\begin{example}\label{ex:notation}
For \(S=(2,3,5,9,15)\), the difference triangle is
\[
\begin{array}{rrrrr}
2 & 3 & 5 & 9 & 15 \\
& 1 & 2 & 4 & 6 \\
& & 1 & 2 & 2 \\
& & & 1 & 0 \\
& & & & 1
\end{array}
\]
The left diagonal below the top row is \((1,1,1,1)\), so \(S\) is
Gilbreath. The right anti-diagonal is
\[
(e_1,e_2,e_3,e_4)=(6,2,0,1),
\]
and \(A(S)=9\). Figure~\ref{fig:anatomy} highlights these two parts of
the triangle.
\end{example}

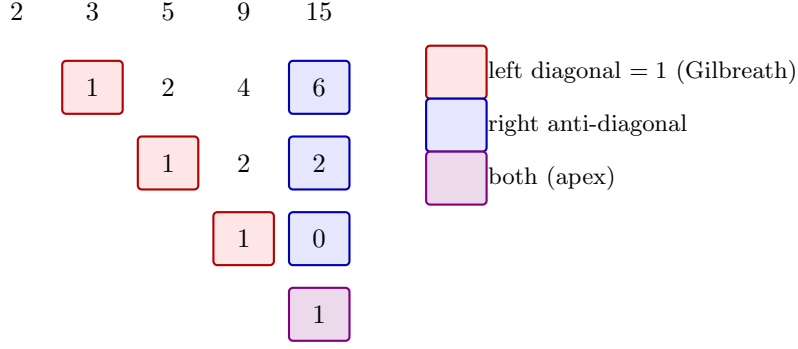
\begin{figure}[htbp]
\centering
\begin{tikzpicture}[
    cell/.style={minimum width=8mm, minimum height=7mm, font=\small, inner sep=2pt},
    leftbox/.style={cell, draw=red!70!black, thick, fill=red!10, rounded corners=1.5pt},
    diagbox/.style={cell, draw=blue!70!black, thick, fill=blue!10, rounded corners=1.5pt},
    bothbox/.style={cell, draw=violet, thick, fill=violet!15, rounded corners=1.5pt},
]
\node[cell] at (0,4) {$2$};
\node[cell] at (1,4) {$3$};
\node[cell] at (2,4) {$5$};
\node[cell] at (3,4) {$9$};
\node[cell] at (4,4) {$15$};

\node[leftbox] at (1,3) {$1$};
\node[cell]    at (2,3) {$2$};
\node[cell]    at (3,3) {$4$};
\node[diagbox] at (4,3) {$6$};

\node[leftbox] at (2,2) {$1$};
\node[cell]    at (3,2) {$2$};
\node[diagbox] at (4,2) {$2$};

\node[leftbox] at (3,1) {$1$};
\node[diagbox] at (4,1) {$0$};

\node[bothbox] at (4,0) {$1$};

\begin{scope}[shift={(5.4,2.5)}]
  \node[leftbox, anchor=west] at (0, 0.7) {};
  \node[anchor=west, font=\footnotesize] at (0.7, 0.7) {left diagonal $=1$ (Gilbreath)};
  \node[diagbox, anchor=west] at (0, 0) {};
  \node[anchor=west, font=\footnotesize] at (0.7, 0) {right anti-diagonal};
  \node[bothbox, anchor=west] at (0, -0.7) {};
  \node[anchor=west, font=\footnotesize] at (0.7, -0.7) {both (apex)};
\end{scope}
\end{tikzpicture}
\caption{Anatomy of the difference triangle of \(S=(2,3,5,9,15)\). The
left diagonal (red) consists of \(1\)'s, which is the defining
Gilbreath condition. The right anti-diagonal (blue) is
\((e_1,e_2,e_3,e_4)=(6,2,0,1)\). When a proposed extension \(k\) is
appended, the new right-edge entries are computed by comparing the
previous new entry with these anti-diagonal entries. The bottom apex
\(1\) belongs to both structures, since \(e_{n-1}=s_1^{n-1}\).}
\label{fig:anatomy}
\end{figure}

The right anti-diagonal determines a composition of folding maps. Define
\[
F_S:\Z_{\geq 0}\to\Z_{\geq 0},\qquad
F_S(d):=\left|\cdots\left||d-e_1|-e_2\right|\cdots-e_{n-1}\right|.
\]
Here the entries \(e_1,e_2,\ldots,e_{n-1}\) are applied in their fixed
anti-diagonal order. This order is part of the structure; unlike the
signed-sum relaxation studied later, the fold parameters cannot be
sorted or chosen independently.

Given a proposed extension \(k\), set
\[
r_0:=|k-s_n|
\]
and then recursively
\[
r_i:=|r_{i-1}-e_i|,\qquad 1\leq i\leq n-1.
\]
Thus \(r_0\) is the new entry created in row \(1\) after appending \(k\),
\(r_1\) is the new entry created in row \(2\), and so on. Equivalently,
if \(d=|k-s_n|\), then
\[
F_S(d)=r_{n-1}.
\]
In particular, \(F_S(|k-s_n|)\) is the new bottom entry of the extended
triangle.

The following criterion is essentially bookkeeping: appending \(k\) only
creates one new right-edge entry in each row, and those entries are
exactly the \(r_i\)'s.

\begin{proposition}[Iterated absolute-value criterion]\label{prop:criterion}
Let \(S\in\G_n\) and \(k\in\Z\). Then
\[
k\in K_S
\quad\Longleftrightarrow\quad
F_S(|k-s_n|)=1.
\]
Equivalently, with \(r_0,
\ldots,r_{n-1}\) defined as above,
\[
k\in K_S
\quad\Longleftrightarrow\quad
r_{n-1}=1.
\]
\end{proposition}

\begin{proof}
Appending \(k\) to \(S\) creates one new entry on the right side of each
row of the difference triangle. The new entry in row \(1\) is
\[
r_0=|k-s_n|.
\]
If the new entry in row \(i\) is \(r_{i-1}\), then the old rightmost
entry in that row is \(e_i\), so the new entry in the next row is
\[
|r_{i-1}-e_i|=r_i.
\]
Therefore \(r_{n-1}=F_S(|k-s_n|)\) is exactly the new bottom entry of the
extended triangle. Since all old entries are unchanged, the extended
sequence is Gilbreath if and only if this new bottom entry is \(1\).
\end{proof}

It is useful to record the corresponding fiber interpretation. If
\[
D_S:=\{|k-s_n|:k\in K_S\}
\]
is the valid distance set, then Proposition~\ref{prop:criterion} gives
\[
D_S=\{d\in\Z_{\geq 0}:F_S(d)=1\}.
\]
Thus the extension problem is an inverse problem for a finite composition
of folding maps \(x\mapsto |x-e_i|\). The reverse-tree algorithm in
Section~\ref{sec:reverse} computes this fiber exactly.

\begin{corollary}[Candidate bound]\label{cor:candidate}
If \(k\in K_S\), then
\[
|k-s_n|\leq A(S)+1.
\]
\end{corollary}

\begin{proof}
Let
\[
d:=|k-s_n|.
\]
Suppose
\[
d>A(S)+1=e_1+\cdots+e_{n-1}+1.
\]
We show that no sign flip occurs while computing the \(r_i\)'s. Since
\(r_0=d\), the claim is true at the beginning. If
\[
r_{i-1}=d-(e_1+\cdots+e_{i-1}),
\]
then
\[
r_{i-1}
>
e_i+e_{i+1}+\cdots+e_{n-1}+1
\geq e_i.
\]
Hence
\[
r_i=|r_{i-1}-e_i|=r_{i-1}-e_i
=d-(e_1+\cdots+e_i).
\]
By induction,
\[
r_{n-1}=d-(e_1+\cdots+e_{n-1})=d-A(S)>1.
\]
Thus \(F_S(d)=r_{n-1}\neq 1\), so \(k\notin K_S\). Taking the
contrapositive gives the desired bound.
\end{proof}

\section{Parity}\label{sec:parity}

The right-adjusted display of the difference triangle is most naturally
read along diagonals rather than columns. For fixed \(a\), the entries
\[
s_a,\ s_a^1,\ s_a^2,\ \ldots,\ s_a^{n-a}
\]
form one diagonal of the triangle. The first diagonal is special:
\(s_1=2\) is even, while the Gilbreath condition says
\[
s_1^1=s_1^2=\cdots=s_1^{n-1}=1.
\]
Thus the first diagonal has parity pattern
\[
\text{even},\ \text{odd},\ \text{odd},\ \ldots,\ \text{odd}.
\]

\begin{lemma}\label{lem:parity}
For every \(S=(s_1,\ldots,s_n)\in\G_n\), every term \(s_a\) with
\(a\geq 2\) is odd, and every positive-row entry \(s_a^b\) with
\(a\geq 2\) and \(b\geq 1\) is even. Equivalently, every diagonal after
the first begins with an odd entry and then consists entirely of even
entries. In particular,
\[
e_1,e_2,\ldots,e_{n-2}
\]
are even, while
\[
e_{n-1}=1.
\]
\end{lemma}

We note that this lemma is also stated in~\cite{gatti2023}. We include
a proof for completeness.

\begin{proof}
We work modulo \(2\). Since signs and absolute values do not matter
modulo \(2\), the recurrence
\[
s_a^b=\left|s_{a+1}^{b-1}-s_a^{b-1}\right|
\]
becomes
\[
s_a^b\equiv s_a^{b-1}+s_{a+1}^{b-1}\pmod 2.
\]
Equivalently,
\[
s_{a+1}^{b-1}\equiv s_a^b+s_a^{b-1}\pmod 2.
\]
Thus, once one diagonal
\[
s_a,\ s_a^1,\ s_a^2,\ldots
\]
is known modulo \(2\), the next diagonal is obtained by adding adjacent
entries on that diagonal.

The first diagonal is known: \(s_1=2\) is even, and the Gilbreath
condition gives
\[
s_1^1=s_1^2=\cdots=s_1^{n-1}=1.
\]
Thus the first diagonal has parity pattern
\[
\text{even},\ \text{odd},\ \text{odd},\ldots,\text{odd}.
\]
For a triangle of size \(5\), the resulting parity pattern is
\[
\begin{array}{ccccc}
E & O & O & O & O \\
& O & E & E & E \\
& & O & E & E \\
& & & O & E \\
& & & & O
\end{array}
\]
where \(E\) denotes even and \(O\) denotes odd.

The general case follows by the same propagation. First, the second
diagonal has the desired pattern, since
\[
s_2\equiv s_1+s_1^1\equiv E+O\equiv O,
\]
while for \(b\geq 1\),
\[
s_2^b\equiv s_1^{b+1}+s_1^b\equiv O+O\equiv E.
\]
Now suppose some diagonal \(a\geq 2\) begins with an odd entry and has
only even entries below it:
\[
s_a\equiv 1\pmod 2,
\qquad
s_a^b\equiv 0\pmod 2\quad (b\geq 1).
\]
Then the next diagonal satisfies
\[
s_{a+1}\equiv s_a+s_a^1\equiv 1+0\equiv 1\pmod 2,
\]
so its top entry is odd. For every \(b\geq 1\),
\[
s_{a+1}^b\equiv s_a^{b+1}+s_a^b\equiv 0+0\equiv 0\pmod 2,
\]
so all lower entries are even. By induction, every diagonal after the
first has this pattern.

Finally, \(e_i=s_{n-i}^i\). For \(1\leq i\leq n-2\), the entry
\(e_i\) lies below the top of one of the later diagonals, so it is even.
The last anti-diagonal entry is
\[
e_{n-1}=s_1^{n-1}=1
\]
by the Gilbreath condition.
\end{proof}

\begin{corollary}\label{cor:parity}
For every \(S\in\G_n\), every \(k\in K_S\) satisfies
\[
k\equiv s_n\pmod 2.
\]
In the normalization \(s_1=2\), every \(k\in K_S\) is odd.
\end{corollary}

\begin{proof}
Let \(k\in K_S\), and define \(r_0,\ldots,r_{n-1}\) as in
Section~\ref{sec:criterion-section}. By Proposition~\ref{prop:criterion},
\[
r_{n-1}=1,
\]
which is odd.

Since
\[
r_{n-1}=|r_{n-2}-e_{n-1}|
\]
and \(e_{n-1}=1\), the value \(r_{n-2}\) must be even. For all earlier
steps, the entries
\[
e_1,e_2,\ldots,e_{n-2}
\]
are even by Lemma~\ref{lem:parity}. Subtracting an even number and
taking an absolute value does not change parity. Therefore the parity
of
\[
r_{n-2},r_{n-3},\ldots,r_0
\]
is the same. In particular, \(r_0\) is even.

But
\[
r_0=|k-s_n|.
\]
Thus \(k-s_n\) is even, so
\[
k\equiv s_n\pmod 2.
\]
Finally, \(s_n\) is odd by Lemma~\ref{lem:parity}, so every
\(k\in K_S\) is odd.
\end{proof}

\section{Candidate set, holes, and defect}\label{sec:defect}

\begin{definition}
The \emph{candidate set} of $S \in \G_n$ is
\[
C_S := \{k \in \Z : |k - s_n| \leq A(S) + 1,\ k \equiv s_n \pmod 2\}.
\]
\end{definition}

By Corollary~\ref{cor:candidate} and Corollary~\ref{cor:parity},
$K_S \subseteq C_S$ always.

\begin{lemma}\label{lem:CS-count}
$|C_S| = A(S) + 2$.
\end{lemma}

\begin{proof}
By Lemma~\ref{lem:parity}, $e_1, \ldots, e_{n-2}$ are even and
$e_{n-1} = 1$, so $A(S)$ is odd and $A(S) + 1$ is even. The interval
$|k - s_n| \leq A(S) + 1$ restricted to $k \equiv s_n \pmod 2$
contains $A(S) + 2$ integers.
\end{proof}

\begin{definition}
The \emph{hole set} of $S$ is $H_S := C_S \setminus K_S$, and the
\emph{defect} is $h(S) := |H_S| = A(S) + 2 - |K_S|$. We say $S$ is
\emph{interval-complete} if $h(S) = 0$ (equivalently, $K_S = C_S$).
\end{definition}

Thus the previously claimed interval characterization is equivalent to
the assertion $h(S) = 0$ for all $S \in \G_n$. The next sections
identify exactly when this holds and the first case where it fails. But first, we examine the signed-sum set proposed by~\cite{gatti2023}.

\section{The signed-sum set}\label{sec:signed-sum}

It is useful to separate two different enlargements of \(K_S\). Let
\(S\in\G_n\) have right anti-diagonal \((e_1,\ldots,e_{n-1})\), and define
the \emph{signed-sum set} originally defined by~\cite{gatti2023} to be 
\[
S_{\pm}
:=
\left\{
s_n+\epsilon_1 e_1+\cdots+\epsilon_{n-1}e_{n-1}+\epsilon_n
:
\epsilon_1,\ldots,\epsilon_n\in\{\pm1\}
\right\}.
\]
This is the set obtained by treating all signs in the unfolded expression
as independent.

Every valid extension lies in this signed-sum set. Indeed, if
\(k\in K_S\), then the chain
\[
r_0=|k-s_n|,\qquad r_i=|r_{i-1}-e_i|
\]
ends at \(r_{n-1}=1\). Unfolding the absolute values along this actual
chain determines a consistent choice of signs, and hence expresses \(k\)
as an element of \(S_{\pm}\). Thus
\[
K_S\subseteq S_{\pm}.
\]
On the other hand, every element of \(S_{\pm}\) has the correct parity
and lies within distance \(A(S)+1\) of \(s_n\), so
\[
S_{\pm}\subseteq C_S.
\]
Therefore
\[
K_S\subseteq S_{\pm}\subseteq C_S.
\]

Both containments can be strict, and they fail for different reasons.
The containment \(S_{\pm}\subseteq C_S\) may be strict because signed
sums need not realize every parity-compatible value in the interval.
The containment \(K_S\subseteq S_{\pm}\) may be strict because an
arbitrary independent choice of signs need not be consistent with the
intermediate values in the nested absolute-value recurrence.

\begin{example}[The signed sums need not fill the candidate interval]\label{ex:signed-sums-miss}
Let
\[
S=(2,3,5,9,15).
\]
Then the right anti-diagonal is
\[
(e_1,e_2,e_3,e_4)=(6,2,0,1),
\]
so
\[
A(S)=9
\]
and
\[
C_S=\{5,7,9,11,13,15,17,19,21,23,25\}.
\]
The signed-sum offsets from \(s_n=15\) are
\[
\pm6\pm2\pm0\pm1\pm1.
\]
These produce
\[
\{-10,-8,-6,-4,-2,2,4,6,8,10\},
\]
but not \(0\). Hence
\[
S_{\pm}
=
\{5,7,9,11,13,17,19,21,23,25\}
=
C_S\setminus\{15\}.
\]
In this example the signed-sum set coincides with the true valid-extension
set:
\[
S_{\pm}=K_S,
\]
but it does not coincide with the full candidate interval \(C_S\).
Thus the interval-filling conclusion does not follow merely from the
existence of a signed-sum expression.
\end{example}

\begin{example}[The signed sums can contain invalid extensions]\label{ex:signs-contain-invalid}
Let
\[
S=(2,3,5,9,17,19).
\]
The right anti-diagonal is
\[
(e_1,e_2,e_3,e_4,e_5)=(2,6,2,0,1),
\]
so \(A(S)=11\), and the candidate set is
\[
C_S=\{7,9,11,13,15,17,19,21,23,25,27,29,31\}.
\]
In this case the signed sums do fill the whole candidate interval:
\[
S_{\pm}=C_S.
\]
However, the true valid-extension set is smaller:
\[
K_S=\{7,9,11,13,15,19,23,25,27,29,31\}.
\]
Thus
\[
S_{\pm}\setminus K_S=\{17,21\}.
\]
The values \(17\) and \(21\) arise from some independent choices of
signs, but those choices are not compatible with the nested
absolute-value recurrence. Therefore the signed-sum set can contain
false positives even when it fills the entire candidate interval.
\end{example}

These examples show that the signed-sum set and the candidate interval
play different roles. The signed-sum set \(S_{\pm}\) is an intermediate
superset of the true valid-extension set \(K_S\), while the candidate
interval \(C_S\) is a still coarser parity-and-size bound. The reverse-tree
algorithm in Section~\ref{sec:reverse} computes \(K_S\) exactly by enforcing the missing consistency conditions.

\section{Signed sums and subset-sum completeness}\label{sec:subset-sum}

The signed-sum set \(S_{\pm}\) has a useful reformulation in the
language of subset sums. This reformulation explains exactly which part
of the extension problem is classical and which part is new. The equality
\(S_{\pm}=C_S\) is a standard complete-sequence question: do the subset
sums of a certain multiset fill the entire interval from \(0\) to their
total sum? The equality \(K_S=C_S\), by contrast, is more rigid: it asks
for the same interval-filling phenomenon under the fixed order imposed by
the nested absolute values.

For a finite multiset \(W=\{w_1,\ldots,w_m\}\) of nonnegative integers,
write
\[
\Sigma(W):=
\left\{\sum_{i\in I}w_i:I\subseteq\{1,\ldots,m\}\right\}
\]
for its set of subset sums, with multiplicities respected. We use the
following classical criterion of Brown~\cite{brown1961}. Zero weights do
not affect \(\Sigma(W)\), so the usual positive-sequence form of
Brown's criterion applies after deleting zeros; we state the equivalent
nonnegative multiset form. Related complete-sequence questions,
including Ramsey-type versions originating with Burr and Erd\H{o}s,
remain active in additive and combinatorial number theory~\cite{burr-erdos,conlon-fox-pham}.

\begin{theorem}[Classical completeness criterion]\label{thm:complete-classical}
Let \(W=\{w_1,\ldots,w_m\}\) be a finite multiset of nonnegative
integers, listed in nondecreasing order
\[
0\leq w_1\leq w_2\leq\cdots\leq w_m,
\]
and let \(T=w_1+\cdots+w_m\). Then
\[
\Sigma(W)=\{0,1,\ldots,T\}
\]
if and only if
\[
w_j\leq 1+\sum_{i<j}w_i
\qquad\text{for every }1\leq j\leq m.
\]
A multiset satisfying this condition will be called \emph{complete}.
\end{theorem}

\begin{proof}
Suppose first that the displayed inequalities hold. We prove by
induction on \(j\) that the subset sums of \(\{w_1,\ldots,w_j\}\) fill
\(\{0,1,\ldots,w_1+\cdots+w_j\}\). The case \(j=0\) is trivial. If the
claim holds through \(j-1\), set \(T_{j-1}=w_1+\cdots+w_{j-1}\). After
adding \(w_j\), the subset sums are the union of
\[
\{0,1,\ldots,T_{j-1}\}
\quad\text{and}\quad
\{w_j,w_j+1,\ldots,w_j+T_{j-1}\}.
\]
The inequality \(w_j\leq T_{j-1}+1\) says exactly that these two
intervals overlap or touch, so their union is the full interval
\(\{0,1,\ldots,T_{j-1}+w_j\}\).

Conversely, suppose \(\Sigma(W)=\{0,1,\ldots,T\}\). If the inequality
failed for some \(j\), then
\[
w_j>1+\sum_{i<j}w_i.
\]
The integer \(1+\sum_{i<j}w_i\) could not be represented as a subset
sum: using only weights before \(w_j\) gives at most \(\sum_{i<j}w_i\),
while using \(w_j\) or any later weight gives at least \(w_j\). This
contradicts completeness. Hence all the inequalities hold.
\end{proof}

Given \(S\in\G_n\) with right anti-diagonal
\((e_1,\ldots,e_{n-1})\), define the associated \emph{weight multiset}
\[
W_S:=\{e_1,e_2,\ldots,e_{n-1},1\}.
\]
Let \(B:=A(S)+1\). Since \(e_{n-1}=1\), the value \(1\) appears in
\(W_S\) at least twice, and
\[
\sum_{w\in W_S}w=B.
\]

\begin{theorem}[Subset-sum reformulation]\label{thm:subset-reformulation}
Let \(S\in\G_n\), and write \(B=A(S)+1\). Then
\[
S_{\pm}=\{s_n-B+2t:t\in\Sigma(W_S)\}.
\]
Consequently,
\[
S_{\pm}=C_S
\quad\Longleftrightarrow\quad
\Sigma(W_S)=\{0,1,\ldots,B\}.
\]
\end{theorem}

\begin{proof}
Write each \(\epsilon_i\in\{-1,+1\}\) uniquely as
\(\epsilon_i=2\delta_i-1\), with \(\delta_i\in\{0,1\}\). Substituting
into the definition of \(S_{\pm}\),
\begin{align*}
s_n+\sum_{i=1}^{n-1}\epsilon_i e_i+\epsilon_n
&=s_n+\sum_{i=1}^{n-1}(2\delta_i-1)e_i+(2\delta_n-1)\\
&=s_n+2\left(\sum_{i=1}^{n-1}\delta_i e_i+\delta_n\right)
-\left(\sum_{i=1}^{n-1}e_i+1\right)\\
&=s_n-B+2t,
\end{align*}
where \(t=\sum_{i=1}^{n-1}\delta_i e_i+\delta_n\) is a subset sum of
\(W_S\). As the \(\delta_i\) range independently over \(\{0,1\}\), the
value \(t\) ranges over \(\Sigma(W_S)\). This proves the first identity.

For the equivalence, note that
\[
C_S=\{s_n-B+2u:u\in\{0,1,\ldots,B\}\},
\]
since \(C_S\) consists of all parity-compatible values in the interval
\([s_n-B,s_n+B]\). Both \(S_{\pm}\) and \(C_S\) are images under the
injective affine map \(x\mapsto s_n-B+2x\), so they coincide if and only
if \(\Sigma(W_S)=\{0,1,\ldots,B\}\).
\end{proof}

\begin{corollary}\label{cor:signed-sum-completeness}
The equality \(S_{\pm}=C_S\) holds if and only if \(W_S\) is complete in
the sense of Theorem~\ref{thm:complete-classical}.
\end{corollary}

Theorem~\ref{thm:subset-reformulation} isolates the classical part of
the problem. The signed-sum set forgets the order in which the absolute
values are evaluated: it only remembers the multiset of fold sizes. By
Brown's criterion, the question \(S_{\pm}=C_S\) is answered by sorting the
weights in \(W_S\) and checking whether each new sorted weight is at most
one plus the sum of the preceding sorted weights.

The true valid-extension set is different. A signed expression represents
a genuine element of \(K_S\) only if its signs arise from an actual chain
\[
r_i=|r_{i-1}-e_i|.
\]
This chain processes the anti-diagonal entries in their fixed geometric
order. Thus the interval-completeness condition for \(K_S\) has the same
``no gap'' shape as Brown's criterion, but the order is forced:
\[
e_i\leq 1+\sum_{j>i}e_j.
\]
In short, Brown's condition is a sorted subset-sum completeness
criterion, while Theorem~\ref{thm:criterion} below is an ordered folding
completeness criterion.

\begin{proposition}[Hierarchy of completeness conditions]\label{prop:hierarchy}
For every \(S\in\G_n\), if \(K_S=C_S\), then \(W_S\) is complete
(equivalently, \(S_{\pm}=C_S\)). The converse fails.
\end{proposition}

\begin{proof}
If \(K_S=C_S\), then the inclusion chain
\(K_S\subseteq S_{\pm}\subseteq C_S\) forces \(S_{\pm}=C_S\). By
Corollary~\ref{cor:signed-sum-completeness}, this is equivalent to
completeness of \(W_S\).

For the converse, take
\[
S=(2,3,5,9,17,19),
\]
as in Example~\ref{ex:signs-contain-invalid}. Its right anti-diagonal is
\[
(e_1,e_2,e_3,e_4,e_5)=(2,6,2,0,1),
\]
so
\[
W_S=\{2,6,2,0,1,1\}.
\]
Sorted, this is \((0,1,1,2,2,6)\). The cumulative sums are
\[
0,1,2,4,6,12,
\]
and each entry is at most one plus the sum of the preceding entries.
Thus \(W_S\) is complete, so \(S_{\pm}=C_S\). However, the ordered
criterion of Theorem~\ref{thm:criterion} fails at \(i=2\), since
\[
e_2=6>1+e_3+e_4+e_5=1+2+0+1=4.
\]
Therefore \(K_S\neq C_S\). Indeed, Example~\ref{ex:signs-contain-invalid}
computes \(S_{\pm}=C_S\) but \(K_S=C_S\setminus\{17,21\}\).
\end{proof}

The classical completeness criterion governs when the signed sums fill
the candidate interval. The ordered interval-completeness criterion
governs when the true valid-extension set fills the interval. The gap
between the two captures the consistency required by the nested
absolute-value structure: elements of \(S_{\pm}\setminus K_S\) are exactly
values produced by independent sign choices that cannot occur along any
actual folding chain.

\section{Endpoint validity and symmetry}\label{sec:endpoints}

\begin{theorem}[Endpoint validity]\label{thm:endpoints}
For every $S \in \G_n$,
\[
s_n - A(S) - 1 \in K_S
\quad \text{and} \quad
s_n + A(S) + 1 \in K_S.
\]
\end{theorem}

\begin{proof}
Take $d = A(S) + 1 = e_1 + e_2 + \cdots + e_{n-1} + 1$. We prove by
induction on $i$ that
\[
r_i = e_{i+1} + e_{i+2} + \cdots + e_{n-1} + 1
\qquad (0 \leq i \leq n - 1),
\]
where the empty sum (at $i = n-1$) is $0$. For $i = 0$ this is the
definition of $d$. Assuming the formula for $r_{i-1}$, we have
\[
r_{i-1} = e_i + \big(e_{i+1} + \cdots + e_{n-1} + 1\big) > e_i,
\]
since the bracketed remainder is at least $1$. Hence no sign flip
occurs and
\[
r_i = |r_{i-1} - e_i| = r_{i-1} - e_i = e_{i+1} + \cdots + e_{n-1} + 1.
\]
At $i=n-1$ this gives \(r_{n-1}=1\), so both
\(s_n+(A(S)+1)\) and \(s_n-(A(S)+1)\) lie in \(K_S\).
\end{proof}

\begin{theorem}[Reflection symmetry]\label{thm:symmetry}
For every $S \in \G_n$ and every $k \in \Z$,
$k \in K_S$ iff $2 s_n - k \in K_S$. Hence $K_S$, $C_S$, and $H_S$ are
symmetric about $s_n$.
\end{theorem}

\begin{proof}
$k \in K_S$ depends on $k$ only through $|k - s_n|$, which is invariant
under $k \mapsto 2 s_n - k$.
\end{proof}

\subsection{Shift-invariance}\label{sec:shift}

For any integer $c$, the map $S \mapsto S + c$ preserves the entire
difference triangle below row $0$, hence the Gilbreath property and the
anti-diagonal. The correspondence $k \leftrightarrow k + c$ gives a
bijection $K_S \to K_{S+c}$. Consequently $|K_S|$, $|C_S|$, $h(S)$,
the component count, and other purely combinatorial invariants depend
only on the gap sequence and are independent of $s_1$. All results
extend verbatim to the shifted family with initial pair $(a, a+1)$ for
any integer $a$.

\section{A reverse-tree algorithm}\label{sec:reverse}

The iterated absolute-value criterion yields a backward algorithm for
computing $K_S$. Instead of starting with a proposed extension $k$ and
pushing the distance $|k-s_n|$ downward through the absolute values, we
start at the required final value $1$ and compute all possible previous
values.

In Section~\ref{sec:criterion-section}, we viewed the valid distance set
as the fiber \(F_S^{-1}(\{1\})\). The reverse-tree algorithm computes
this fiber by inverting the folds one at a time.

\begin{definition}
For $e \in \Z_{\geq 0}$ and $T \subseteq \Z_{\geq 0}$, define the
\emph{preimage step}
\[
P_e(T) := \{e + t : t \in T\} \cup \{e - t : t \in T,\ e \geq t\}.
\]
\end{definition}

This is exactly the set of nonnegative solutions $x$ to equations of the
form $|x-e|=t$ with $t\in T$. The branch $x=e+t$ is always allowed,
whereas the branch $x=e-t$ is allowed only when $e\geq t$. Thus the
second branch is not $|e-t|$ in general.

\begin{proposition}[Reverse-tree characterization]\label{prop:reverse}
Let $T_{n-1} := \{1\}$ and recursively $T_{i-1} := P_{e_i}(T_i)$ for
$i = n - 1, n - 2, \ldots, 1$. Then
\[
D_S := \{|k - s_n| : k \in K_S\} = T_0,
\]
and $K_S = \{s_n + d : d \in D_S\} \cup \{s_n - d : d \in D_S\}$.
\end{proposition}

\begin{proof}
The condition $d = |k - s_n| \in D_S$ is equivalent to: there exists a
nonnegative chain $r_0 = d, r_1, \ldots, r_{n-1}$ with
$r_i = |r_{i-1} - e_i|$ and $r_{n-1} = 1$. Working backward, if
$r_i = t$ then the possible values of $r_{i-1} \geq 0$ are $e_i + t$
(always) and $e_i - t$ (only when $e_i \geq t$), so the possible
values form exactly $P_{e_i}(\{t\})$. Iterating gives $T_0 = D_S$.
\end{proof}

\begin{example}[Reverse tree for $S=(2,3,5,9,15)$]
Let $S=(2,3,5,9,15)$. Its right anti-diagonal is
$(e_1,e_2,e_3,e_4)=(6,2,0,1)$.
We begin from the required final value $T_4=\{1\}$ and move upward
through the anti-diagonal:
\[
T_3=P_1(\{1\})=\{0,2\}, \quad
T_2=P_0(\{0,2\})=\{0,2\},
\]
since the lower branch $0-2$ is not allowed,
\[
T_1=P_2(\{0,2\})=\{0,2,4\}, \quad
T_0=P_6(\{0,2,4\})=\{2,4,6,8,10\}.
\]
Thus $D_S=T_0=\{2,4,6,8,10\}$, and reflecting these distances around
$s_n=15$ gives
\[
K_S = \{15\pm d : d\in D_S\} = \{5,7,9,11,13,17,19,21,23,25\}.
\]
This process is visualized in the figure below.
\end{example}

\begin{figure}[htbp]
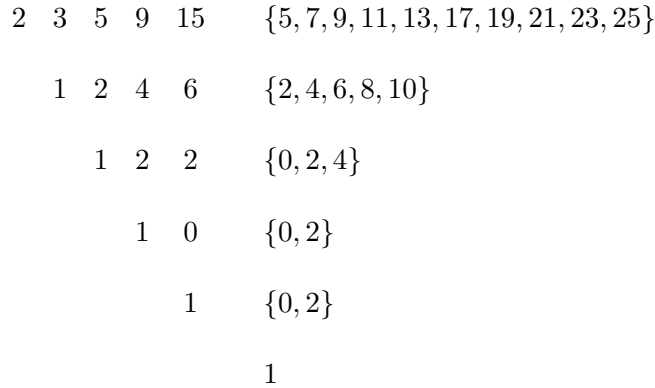

\centering
\[
\begin{array}{ccccc@{\qquad}l}
2 & 3 & 5 & 9 & 15 & \{5,7,9,11,13,17,19,21,23,25\} \\[1.2em]
  & 1 & 2 & 4 & 6  & \{2,4,6,8,10\} \\[1.2em]
  &   & 1 & 2 & 2  & \{0,2,4\} \\[1.2em]
  &   &   & 1 & 0  & \{0,2\} \\[1.2em]
  &   &   &   & 1  & \{0,2\} \\[1.2em]
  &   &   &   &    & 1
\end{array}
\]
\caption{The reverse-tree process on the sequence $(2,3,5,9,15)$.}
\label{fig:reverse-tree-235915}
\end{figure}

\section{Interval-complete sequences}\label{sec:criterion-theorem}

We now characterize exactly when $K_S$ equals the full candidate
interval $C_S$. By Corollary~\ref{cor:signed-sum-completeness}, the
weaker equality $S_{\pm}=C_S$ is controlled by the classical
subset-sum completeness of the sorted multiset $W_S$. The equality
$K_S=C_S$ is more rigid. It requires the independently signed expression
to be compatible with the ordered folding recurrence
\(r_i=|r_{i-1}-e_i|\), or equivalently with the fiber condition
\(F_S(d)=1\). The criterion below is therefore an ordered folding
analogue of Brown's completeness criterion. Figure~\ref{fig:criterion-mech}
illustrates the local mechanism that underlies the criterion.

\begin{figure}[htbp]
\centering
\begin{tikzpicture}[scale=0.6, every node/.style={font=\footnotesize}]
\node[anchor=west, font=\small] at (-0.5, 3.5) {\textbf{(a)} $a \le L$: $Q_a(\widetilde{T}) = \{0, 1, \ldots, a+L\}$ is the full interval.};
\draw[->, thick] (-0.5, 0) -- (11, 0);
\draw (0, 0.12) -- (0, -0.12); \node[below=2pt] at (0, 0) {$0$};
\draw (3, 0.12) -- (3, -0.12); \node[below=2pt] at (3, 0) {$a$};
\draw (6, 0.12) -- (6, -0.12); \node[below=2pt] at (6, 0) {$L$};
\draw (9, 0.12) -- (9, -0.12); \node[below=2pt] at (9, 0) {$a+L$};
\draw[red!70!black, line width=2.5pt] (0, 0.9) -- (3, 0.9);
\node[anchor=west, font=\scriptsize, color=red!70!black] at (3.2, 0.9) {lower: $a - \widetilde{T} = [0, a]$};
\draw[blue!70!black, line width=2.5pt] (3, 1.6) -- (9, 1.6);
\node[anchor=west, font=\scriptsize, color=blue!70!black] at (9.2, 1.6) {upper: $a + \widetilde{T} = [a, a+L]$};
\draw[green!50!black, line width=3pt] (0, 2.3) -- (9, 2.3);
\node[anchor=west, font=\scriptsize, color=green!50!black] at (9.2, 2.3) {$Q_a(\widetilde{T}) = \{0, \ldots, a+L\}$};

\begin{scope}[yshift=-5.5cm]
\node[anchor=west, font=\small] at (-0.5, 3.5) {\textbf{(b)} $a > L$: lower branch starts at $a - L > 0$, leaving a gap.};
\draw[->, thick] (-0.5, 0) -- (12.5, 0);
\draw (0, 0.12) -- (0, -0.12); \node[below=2pt] at (0, 0) {$0$};
\draw (5, 0.12) -- (5, -0.12); \node[below=2pt] at (5, 0) {$a - L$};
\draw (8, 0.12) -- (8, -0.12); \node[below=2pt] at (8, 0) {$a$};
\draw (11, 0.12) -- (11, -0.12); \node[below=2pt] at (11, 0) {$a + L$};
\draw[red!70!black, line width=2.5pt] (5, 0.9) -- (8, 0.9);
\node[anchor=west, font=\scriptsize, color=red!70!black] at (8.2, 0.9) {lower: $a - \widetilde{T} = [a-L, a]$};
\draw[blue!70!black, line width=2.5pt] (8, 1.6) -- (11, 1.6);
\node[anchor=west, font=\scriptsize, color=blue!70!black] at (11.2, 1.6) {upper: $a + \widetilde{T} = [a, a+L]$};
\draw[orange!80!black, line width=3pt, dotted] (0, 2.3) -- (4.8, 2.3);
\node[anchor=west, font=\scriptsize, color=orange!80!black] at (5, 2.3) {missing: $\{0, \ldots, a-L-1\}$};
\end{scope}
\end{tikzpicture}
\caption{The mechanism behind the interval-completeness criterion
(Theorem~\ref{thm:criterion}). With $\widetilde{T} = \{0,1,\ldots,L\}$, the
preimage map $Q_a(\widetilde{T})$ consists of a lower branch $a - \widetilde{T}$ and an upper
branch $a + \widetilde{T}$. (a) When $a \le L$, the two branches meet at $a$ and
together cover the full integer interval $\{0,1,\ldots,a+L\}$. (b)
When $a > L$, the lower branch starts at $a - L > 0$ and the values
$\{0, 1, \ldots, a - L - 1\}$ are missing from $Q_a(\widetilde{T})$. The criterion
$e_i \le 1 + \sum_{j>i} e_j$ ensures that case~(a) occurs at every
step of the reverse tree.}
\label{fig:criterion-mech}
\end{figure}
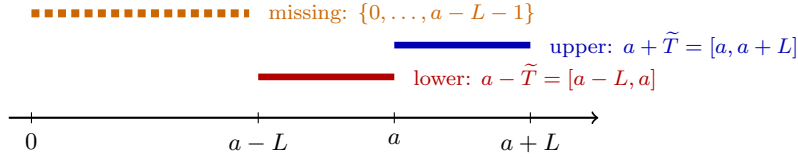

\begin{theorem}[Interval-completeness criterion]\label{thm:criterion}
Let $S \in \G_n$ with $n \geq 2$ and right anti-diagonal
$(e_1, \ldots, e_{n-1})$. Then $K_S = C_S$ if and only if
\[
e_i \leq 1 + \sum_{j=i+1}^{n-1} e_j
\qquad \text{for every } 1 \leq i \leq n - 2.
\]
\end{theorem}

\begin{proof}
For $n = 2$ the index range $1 \leq i \leq n-2$ is empty, the condition
holds vacuously, and indeed $\G_2 = \{(2,3)\}$ has
$K_S = C_S = \{1,3,5\}$; we therefore assume $n \geq 3$.

We use the reverse-tree characterization of
Proposition~\ref{prop:reverse}. Recall $T_{n-1} = \{1\}$, and
$T_{i-1} = P_{e_i}(T_i)$ for $i = n - 1, \ldots, 1$, with
$D_S = T_0$. By Lemma~\ref{lem:parity}, $e_{n-1} = 1$ and
$e_1, \ldots, e_{n-2}$ are even. The first reverse step gives
\[
T_{n-2} = P_1(\{1\}) = \{0, 2\}.
\]

From this stage onward, all elements of $T_i$ are even, since they
arise by adding or subtracting an even $e_j$ from even values. We
normalize by dividing by $2$: write $a_i := e_i / 2$ for
$1 \leq i \leq n - 2$, and set
\[
\widetilde{T}_{n-2} := T_{n-2} / 2 = \{0, 1\},
\qquad \widetilde{T}_{i-1} := Q_{a_i}(\widetilde{T}_i),
\]
where $Q_a(\widetilde{T}) := \{a + u : u \in \widetilde{T}\} \cup \{a - u : u \in \widetilde{T},\ a \geq u\}$
is the normalized preimage map. Then $D_S = 2 \widetilde{T}_0$.

For $0 \leq i \leq n - 2$, define
$L_i := 1 + \sum_{j=i+1}^{n-2} a_j$, so $L_{n-2} = 1$ and
$L_{i-1} = a_i + L_i$. The full candidate distance set (after
normalization) corresponds to $\{0, 1, \ldots, L_0\}$: indeed,
$|C_S| = A(S) + 2$ and the candidate distances are
$\{0, 2, 4, \ldots, A(S) + 1\}$, normalizing to
$\{0, 1, \ldots, (A(S)+1)/2\} = \{0, 1, \ldots, L_0\}$ since
$L_0 = 1 + (a_1 + \cdots + a_{n-2}) = 1 + (A(S) - 1)/2 = (A(S) + 1)/2$.

Therefore $K_S = C_S$ is equivalent to $\widetilde{T}_0 = \{0, 1, \ldots, L_0\}$.
The proof reduces to the following elementary claim, whose content is
exactly Figure~\ref{fig:criterion-mech}.

\medskip
\textbf{Claim.} Let $\widetilde{T} \subseteq \{0, 1, \ldots, L\}$ and $a \geq 0$.
Then $Q_a(\widetilde{T}) = \{0, 1, \ldots, a + L\}$ if and only if
$\widetilde{T} = \{0, 1, \ldots, L\}$ and $a \leq L$.
\medskip

\textit{Proof of claim.}
$(\Leftarrow)$ Suppose $\widetilde{T} = \{0, 1, \ldots, L\}$ and $a \leq L$. Then
$\{a + u : u \in \widetilde{T}\} = \{a, a+1, \ldots, a + L\}$, and since
$a \leq L$, every $u \in \widetilde{T}$ with $u \leq a$ is in $\widetilde{T}$, so
$\{a - u : u \in \widetilde{T}, a \geq u\} = \{0, 1, \ldots, a\}$. The union is
$\{0, 1, \ldots, a + L\}$.

$(\Rightarrow)$ Suppose $Q_a(\widetilde{T})=\{0,1,\ldots,a+L\}$. Any
$v>a$ in $Q_a(\widetilde{T})$ can only arise as $v=a+u$ for some $u\in \widetilde{T}$
(since $a-u\leq a$). Hence for every $1\leq u\leq L$, the value
$a+u$ being in $Q_a(\widetilde{T})$ forces $u\in \widetilde{T}$. So
$\{1,\ldots,L\}\subseteq \widetilde{T}$. Also, $a\in Q_a(\widetilde{T})$ requires
$0\in \widetilde{T}$ (via $a+0$ or $a-0$). Thus $\widetilde{T}=\{0,1,\ldots,L\}$.

It remains to show $a\leq L$. If $a=0$, this is immediate. If
$a>0$, then the value $0\in Q_a(\widetilde{T})$ cannot arise from the upper
branch $a+u$, so it must arise from the lower branch $a-u=0$ for
some $u\in \widetilde{T}$. Hence $a=u\in \widetilde{T}$. Since
$\widetilde{T}\subseteq\{0,1,\ldots,L\}$, this forces $a\leq L$.
\hfill$\square$

\medskip

\emph{Boundedness.} Before applying the claim we record that
$\widetilde{T}_i \subseteq \{0, 1, \ldots, L_i\}$ for every $i$. This holds at
$i = n - 2$ since $\widetilde{T}_{n-2} = \{0, 1\}$ and $L_{n-2} = 1$; and if
$\widetilde{T}_i \subseteq \{0, \ldots, L_i\}$, then every element of
$Q_{a_i}(\widetilde{T}_i)$ has the form $a_i + u$ or $a_i - u$ with
$0 \leq u \leq L_i$, hence lies in $[0,\, a_i + L_i] = [0,\, L_{i-1}]$.
Thus $\widetilde{T}_{i-1} \subseteq \{0, \ldots, L_{i-1}\}$, completing the
induction.

\medskip

\emph{Sufficiency.} Suppose $e_i \leq 1 + \sum_{j > i} e_j$ for every
$1 \leq i \leq n - 2$. Since $e_{n-1} = 1$, this rearranges to
$e_i \leq 2 + \sum_{j=i+1}^{n-2} e_j$, and dividing by $2$ gives
$a_i \leq 1 + \sum_{j=i+1}^{n-2} a_j = L_i$. By the claim
$(\Leftarrow)$ applied at each step $i = n - 2, n - 3, \ldots, 1$, the
equality $\widetilde{T}_{i-1} = \{0, 1, \ldots, L_{i-1}\}$ propagates from
$\widetilde{T}_{n-2}$ down to $\widetilde{T}_0 = \{0, 1, \ldots, L_0\}$. Hence $D_S = 2 \widetilde{T}_{0}$ is
the full parity-compatible interval, and $K_S = C_S$.

\emph{Necessity.} Conversely, suppose $K_S = C_S$, equivalently
$\widetilde{T}_0 = \{0, 1, \ldots, L_0\}$. Since $\widetilde{T}_0 = Q_{a_1}(\widetilde{T}_1)$ and, by the
boundedness established above, $\widetilde{T}_1 \subseteq \{0, 1, \ldots, L_1\}$,
the claim $(\Rightarrow)$ forces $\widetilde{T}_1 = \{0, 1, \ldots, L_1\}$ and
$a_1 \leq L_1$. Iterating, $\widetilde{T}_i = \{0, 1, \ldots, L_i\}$ and
$a_i \leq L_i$ for every $1 \leq i \leq n - 2$. In unnormalized form,
$e_i \leq 2 L_i = 2 + \sum_{i < j \leq n-2} e_j$, equivalently
$e_i \leq 1 + \sum_{j > i} e_j$ (using $e_{n-1} = 1$).
\end{proof}

\begin{remark}[Generality of the criterion]\label{rem:general-folding}
The proof of Theorem~\ref{thm:criterion} uses the Gilbreath assumption
only through Lemma~\ref{lem:parity} (which gives \(e_{n-1}=1\) and
\(e_1,\ldots,e_{n-2}\) even). The criterion therefore applies to any
ordered tuple of nonnegative integers
\((e_1,\ldots,e_{m-1})\) with \(e_{m-1}=1\) and the remaining entries
even, regardless of whether \((e_i)\) arises as the right anti-diagonal
of a Gilbreath sequence. In this generality, with
\(F(d)=||\cdots ||d-e_1|-e_2|\cdots -e_{m-1}|\), the fiber
\(F^{-1}(\{1\})\) coincides with the full parity-compatible interval
\(\{0,2,\ldots,A+1\}\), \(A=\sum_i e_i\), if and only if
\(e_i\leq 1+\sum_{j>i}e_j\) for every \(1\leq i\leq m-2\).
\end{remark}

\begin{corollary}\label{cor:ic-AP}
The minimal sequence
\(L_n=(2,3,5,7,\ldots,2n-1)\)
is interval-complete.
\end{corollary}

\begin{proof}
Its right anti-diagonal is $(2, 0, 0, \ldots, 0, 1)$. The only
nontrivial criterion is at $i = 1$: $e_1 = 2 \leq 1 + 0 + \cdots + 0 + 1
= 2$.
\end{proof}

\begin{corollary}\label{cor:ic-doubling}
The doubling sequence $U_n = (2, 3, 5, 9, 17, \ldots, 2^{n-1} + 1)$
is interval-complete, so $|K_{U_n}| = A(U_n) + 2 = 2^{n-1} + 1$.
\end{corollary}

\begin{proof}
For every positive row $b\geq 1$, the triangle of $U_n$ has row $b$
equal to
\[
(1,2,4,\ldots,2^{n-b-1}).
\]
Thus the right anti-diagonal is
\[
(e_1,\ldots,e_{n-1})=(2^{n-2},2^{n-3},\ldots,2,1).
\]
At each $i$, we have $e_i=2^{n-i-1}$, while
\[
1+\sum_{j>i}e_j
=
1+(2^{n-i-1}-1)
=
2^{n-i-1}.
\]
Equality holds throughout the criterion, so $U_n$ is interval-complete.
Since
\[
A(U_n)=2^{n-1}-1,
\]
we get
\[
|K_{U_n}|=A(U_n)+2=2^{n-1}+1.
\]
\end{proof}

\section{The first hole}\label{sec:first-hole}

\begin{figure}[htbp]
\centering
\begin{tikzpicture}[scale=0.55, every node/.style={font=\footnotesize}]
\draw[->, thick] (3.5, 0) -- (26.5, 0);
\foreach \x in {5,7,9,11,13,15,17,19,21,23,25}{
  \draw (\x, 0.15) -- (\x, -0.15);
  \node[below=2pt] at (\x, -0.15) {$\x$};
}
\foreach \x in {5,7,9,11,13,17,19,21,23,25}{
  \filldraw[blue!70!black] (\x, 0) circle (0.16);
}
\filldraw[fill=white, draw=red, thick] (15, 0) circle (0.18);
\draw[red, thick] (14.7, -0.3) -- (15.3, 0.3);
\draw[red, thick] (14.7, 0.3) -- (15.3, -0.3);
\node[above=10pt, color=red!70!black] at (15, 0.2) {hole at $k = s_n = 15$};
\draw[<-, thick] (5, 1.3) -- (5, 0.3);
\node[above] at (5, 1.3) {$s_n - A(S) - 1$};
\draw[<-, thick] (25, 1.3) -- (25, 0.3);
\node[above] at (25, 1.3) {$s_n + A(S) + 1$};
\draw[thick, gray] (5, -1.0) -- (5, -0.7) -- (25, -0.7) -- (25, -1.0);
\node[below, color=gray!50!black] at (15, -1.1) {$C_S$: 11 odd integers from $5$ to $25$};
\end{tikzpicture}
\caption{The first hole. For $S = (2, 3, 5, 9, 15)$, the candidate set
$C_S$ consists of all odd integers in $[s_n - A(S) - 1,\, s_n + A(S) + 1]
= [5, 25]$. The valid-extension set $K_S$ (blue dots) contains all of
these except the center value $k = 15$ (marked $\times$). Thus
$H_S = \{15\}$ and $h(S) = 1$.}
\label{fig:first-hole}
\end{figure}
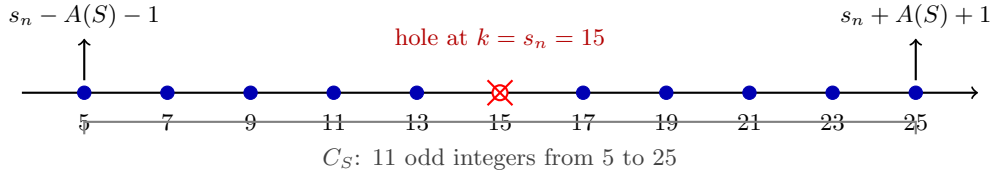

\begin{theorem}[First hole]\label{thm:hole}
For $n\leq 4$, every $S\in\G_n$ has $K_S=C_S$. The smallest $n$
for which some $S\in\G_n$ has $K_S\neq C_S$ is $n=5$. At this length,
the unique sequence $S\in\G_5$ with $K_S\neq C_S$ is
\[
S=(2,3,5,9,15).
\]
Explicitly,
\[
C_S=\{5,7,9,11,13,15,17,19,21,23,25\},
\quad
K_S=\{5,7,9,11,13,17,19,21,23,25\},
\]
so $H_S=\{15\}$ (Figure~\ref{fig:first-hole}).
\end{theorem}

\begin{proof}
\emph{The case $n \leq 4$.} The families are
\[
\G_2=\{(2,3)\},\qquad
\G_3=\{(2,3,5)\},\qquad
\G_4=\{(2,3,5,7),(2,3,5,9)\}.
\]
Their right anti-diagonals are respectively
\[
(1),\qquad (2,1),\qquad (2,0,1),\qquad (4,2,1).
\]
Each satisfies the criterion of Theorem~\ref{thm:criterion}
(vacuously for $n=2$, and directly for the others), so all four
sequences are interval-complete.

\emph{The case $n = 5$.} $\G_5$ has six sequences. These are obtained by extending the two elements of \(\G_4\):
\[
K_{(2,3,5,7)}^+=\{9,11\},\qquad
K_{(2,3,5,9)}^+=\{11,13,15,17\}.
\]
Thus the displayed six sequences are all of \(\G_5\). We list them with
their right anti-diagonals:
\[
\begin{array}{l|l}
S & (e_1, e_2, e_3, e_4) \\
\hline
(2,3,5,7,9) & (2, 0, 0, 1) \\
(2,3,5,7,11) & (4, 2, 2, 1) \\
(2,3,5,9,11) & (2, 2, 0, 1) \\
(2,3,5,9,13) & (4, 0, 2, 1) \\
(2,3,5,9,15) & (6, 2, 0, 1) \\
(2,3,5,9,17) & (8, 4, 2, 1)
\end{array}
\]
For each of these we check the criterion \(e_i \leq 1+\sum_{j>i}e_j\)
at \(i=1,2,3\). All five sequences except \((2,3,5,9,15)\) satisfy the
criterion at every \(i\) and are therefore interval-complete by
Theorem~\ref{thm:criterion}. For \((2,3,5,9,15)\), the criterion fails
at \(i=1\): \(e_1=6>1+2+0+1=4\). The corresponding extension sets are
\[
\begin{array}{l|l|c}
S & K_S & |K_S| \\
\hline
(2,3,5,7,9) & \{5,7,9,11,13\} & 5 \\
(2,3,5,7,11) & \{1,3,5,7,9,11,13,15,17,19,21\} & 11 \\
(2,3,5,9,11) & \{5,7,9,11,13,15,17\} & 7 \\
(2,3,5,9,13) & \{5,7,9,11,13,15,17,19,21\} & 9 \\
(2,3,5,9,15) & \{5,7,9,11,13,17,19,21,23,25\} & 10 \\
(2,3,5,9,17) & \{1,3,5,\ldots,33\} & 17
\end{array}
\]
and only $(2,3,5,9,15)$ has $|K_S| \neq A(S) + 2$.

\emph{Identifying the hole.} The right anti-diagonal of
$S = (2,3,5,9,15)$ is $(6, 2, 0, 1)$ (Example~\ref{ex:notation}).
Applying the reverse tree of Proposition~\ref{prop:reverse}:
\[
T_4 = \{1\},\
T_3 = P_1(\{1\}) = \{0, 2\},\
T_2 = P_0(\{0, 2\}) = \{0, 2\},
\]
\[
T_1 = P_2(\{0, 2\}) = \{0, 2, 4\},\
T_0 = P_6(\{0, 2, 4\}) = \{2, 4, 6, 8, 10\}.
\]
The value $0$ would be in $T_0$ only via $6 - 6$, but $6 \notin T_1$;
hence $0 \notin T_0$. Therefore $D_S = \{2, 4, 6, 8, 10\}$, giving
$K_S = \{15 \pm d : d \in D_S\} = C_S \setminus \{15\}$, so
$H_S = \{15\}$.
\end{proof}

\section{Minimum extension width}\label{sec:min}

\begin{theorem}[Minimum extension width]\label{thm:min}
For every $n \geq 3$,
\[
\min_{S \in \G_n} |K_S| = 5,
\]
and the minimum is uniquely achieved by the minimal sequence
\(L_n=(2,3,5,7,\ldots,2n-1)\).
\end{theorem}

We separate the proof into three lemmas.

\begin{lemma}\label{lem:minD}
For every $S \in \G_n$, $|D_S| \geq 3$.
\end{lemma}

\begin{proof}
The reverse tree starts with $T_{n-1} = \{1\}$, and the first step
gives $T_{n-2} = P_1(\{1\}) = \{0, 2\}$. Each preimage step
$P_e(T)$ contains the translated set $\{e + t : t \in T\}$ with the
same cardinality as $T$, so the cardinality of $T_i$ is non-decreasing
as $i$ decreases.

By Lemma~\ref{lem:parity}, $e_1 = s_n - s_{n-1}$ is a positive even
integer, so $e_1 \geq 2$.

\emph{Case 1.} If $e_2 = \cdots = e_{n-2} = 0$, then each step
$P_0(\{0, 2\}) = \{0, 2\}$ leaves the set unchanged, so $T_1 = \{0, 2\}$.
The final step is $P_{e_1}(\{0, 2\})$. Since $e_1 \geq 2$, both
$e_1 - 0 = e_1$ and $e_1 - 2$ are nonnegative, giving
$T_0 = \{e_1 - 2, e_1, e_1 + 2\}$, three distinct values.

\emph{Case 2.} If some $e_j$ with $2 \leq j \leq n - 2$ is positive,
take the largest such $j$. Then $e_{j+1} = \cdots = e_{n-2} = 0$, so
$T_j = \{0, 2\}$. Since $e_j$ is a positive even integer
(Lemma~\ref{lem:parity}), $e_j \geq 2$, and
$T_{j-1} = P_{e_j}(\{0, 2\}) = \{e_j - 2, e_j, e_j + 2\}$ has $3$
distinct elements. By cardinality non-decrease, $|T_0| \geq 3$.

In both cases $|D_S| = |T_0| \geq 3$.
\end{proof}

\begin{lemma}\label{lem:equality}
If $|K_S| = 5$, then $(e_1, e_2, \ldots, e_{n-1}) = (2, 0, 0, \ldots, 0, 1)$.
\end{lemma}

\begin{proof}
By Theorem~\ref{thm:symmetry}, $|K_S| = 2|D_S|$ if $0 \notin D_S$ and
$|K_S| = 2|D_S| - 1$ if $0 \in D_S$. So $|K_S| = 5$ requires
$|D_S| = 3$ and $0 \in D_S$.

For $0 \in D_S = P_{e_1}(T_1)$, we need $e_1 - t = 0$ for some
$t \in T_1$, i.e., $e_1 \in T_1$. Since $e_1 \geq 2$, this means
$T_1$ contains an element $\geq 2$.

If $|T_1| \geq 3$, then $P_{e_1}(T_1)$ contains the three distinct
positive elements $\{e_1 + t : t \in T_1\}$, plus the element $0$ from
the lower branch, giving $|D_S| \geq 4$ and contradicting $|D_S| = 3$.
Hence $|T_1| = 2$.

The sequence $T_{n-2} = \{0, 2\}, T_{n-3}, \ldots, T_1$ has
non-decreasing cardinality, so all these sets have cardinality $2$.
Each middle $e_j$ is even (Lemma~\ref{lem:parity}), so either
$e_j = 0$, in which case $P_{e_j}(\{0,2\}) = \{0,2\}$, or $e_j \geq 2$,
in which case $P_{e_j}(\{0,2\}) = \{e_j - 2,\, e_j,\, e_j + 2\}$ has
three distinct elements. Preservation of cardinality $2$ therefore
forces $e_j = 0$ for each $2 \leq j \leq n - 2$, and $T_1 = \{0, 2\}$.

Since $e_1 \in T_1 = \{0, 2\}$ and $e_1 \geq 2$, $e_1 = 2$. Combined
with $e_{n-1} = 1$, we get $(e_1, \ldots, e_{n-1}) = (2, 0, \ldots, 0, 1)$.
\end{proof}

\begin{lemma}\label{lem:AP}
If $S \in \G_n$ has $(e_1, e_2, \ldots, e_{n-1}) = (2, 0, 0, \ldots, 0, 1)$,
then $S = L_n$.
\end{lemma}

\begin{proof}
Write $g_j := s_{j+1}-s_j$ for the gaps of $S$. Since
$(s_1,s_2)=(2,3)$, we have $g_1=1$. Also
$g_{n-1}=s_n-s_{n-1}=e_1=2$.

We prove by descending induction that
\[
 g_j=2 \qquad (2\leq j\leq n-1).
\]
The base case $j=n-1$ was just proved. Now suppose $2\leq j\leq n-2$
and assume inductively that
\[
 g_{j+1}=g_{j+2}=\cdots=g_{n-1}=2.
\]
Then the row-$1$ entries strictly to the right of position $j$ are all
$2$. Hence, by induction on the row index, every entry $s_a^b$ with
$a\geq j+1$, $b\geq 2$, and $a+b\leq n$ is zero: in row $2$ these are
absolute differences of equal row-$1$ entries, and in higher rows they
are absolute differences of zeros.

Now use the anti-diagonal hypothesis at index $i=n-j$. Since
$e_i=s_{n-i}^i$, we have
\[
 e_{n-j}=s_j^{n-j}=0.
\]
We propagate this zero upward to row $2$. For each
$m=2,3,\ldots,n-j-1$, the entry $s_{j+1}^{m}$ is zero by the preceding
paragraph, and the recurrence gives
\[
 s_j^{m+1}=|s_{j+1}^{m}-s_j^{m}|=|0-s_j^{m}|=s_j^{m}.
\]
If \(n-j=2\), this already says \(s_j^2=0\). Otherwise, applying the
identity successively for
\[
m=n-j-1,\ n-j-2,\ \ldots,\ 2,
\]
gives
\[
s_j^{n-j}=s_j^{n-j-1}=\cdots=s_j^2.
\]
Since \(s_j^{n-j}=0\), it follows in all cases that \(s_j^2=0\).
But
\[
 s_j^2=|s_{j+1}^1-s_j^1|=|g_{j+1}-g_j|=|2-g_j|.
\]
Therefore $g_j=2$, completing the descending induction.

Thus $g_1=1$ and $g_2=\cdots=g_{n-1}=2$, so
\[
 S=(2,3,5,7,\ldots,2n-1)=L_n.
\]
\end{proof}

\begin{proof}[Proof of Theorem~\ref{thm:min}]
By Lemma~\ref{lem:minD}, $|D_S| \geq 3$, so $|K_S| \geq 2 \cdot 3 - 1
= 5$. Equality requires $0 \in D_S$ and $|D_S| = 3$, which by
Lemma~\ref{lem:equality} forces the anti-diagonal
$(2, 0, \ldots, 0, 1)$, which by Lemma~\ref{lem:AP} forces
$S = L_n$. Conversely, $L_n$ has this anti-diagonal,
and the reverse tree gives $D_{L_n} = \{0, 2, 4\}$, hence
$|K_{L_n}| = 5$.
\end{proof}

\section{The doubling sequence and the maximum-width conjecture}\label{sec:max}

\begin{theorem}[Extension width of the doubling sequence]\label{thm:doubling}
For every $n \geq 2$, $|K_{U_n}| = 2^{n-1} + 1$.
\end{theorem}

\begin{proof}
By Corollary~\ref{cor:ic-doubling}.
\end{proof}

\begin{conjecture}[Maximum width]\label{conj:max}
For every $n \geq 2$, $\max_{S \in \G_n} |K_S| = 2^{n-1} + 1$, uniquely
achieved by $U_n$.
\end{conjecture}

This conjecture is verified by exhaustive computation for $n \leq 10$;
see Table~\ref{tab:data}. A proof would have to use the fact that the
anti-diagonal arises from a Gilbreath sequence, not from arbitrary
nonnegative integers, since the reverse-tree analysis alone permits
anti-diagonals exceeding the doubling bound.

\section{An exponentially disconnected family}\label{sec:Vn}

For a finite set $X \subseteq \Z$ all of one parity class,
$\#\textup{comp}(X)$ denotes the number of maximal runs of common
difference $2$.

\begin{definition}
For $n \geq 5$, set $V_n := (v_1, \ldots, v_n)$ where $v_1 = 2$,
$v_i = 2^{i-1} + 1$ for $2 \leq i \leq n - 1$, and $v_n = 2^{n-1} - 1$.
\end{definition}

So $V_n$ follows the doubling through position $n - 1$ and then
undershoots at the last position by $2$: $V_5 = (2,3,5,9,15)$,
$V_6 = (2,3,5,9,17,31)$, $V_7 = (2,3,5,9,17,33,63)$.

\begin{example}\label{ex:V5-V6}
For $V_5$, the reverse tree gives $D_5 = \{2, 4, 6, 8, 10\}$, a single
block of $5$ consecutive even integers. For $V_6$, the right
anti-diagonal is $(14, 6, 2, 0, 1)$; running the reverse tree, the
preimage step $P_{14}$ applied to $D_5$ splits each element into two
preimages $14 - d$ and $14 + d$, giving
$D_6 = \{4, 6, 8, 10, 12\} \cup \{16, 18, 20, 22, 24\}$, two
$5$-element blocks separated by a gap.
Figure~\ref{fig:Vn-doubling} shows the first three stages.
\end{example}

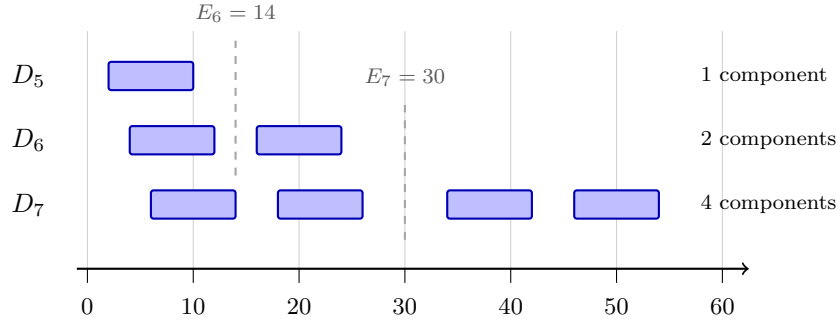
\begin{figure}[htbp]
\centering
\begin{tikzpicture}[
    x=0.14cm,
    y=0.85cm,
    every node/.style={font=\footnotesize},
    block/.style={fill=blue!25, draw=blue!70!black, thick, rounded corners=1pt},
    gridline/.style={gray!35, thin},
]

\def\yaxis{0.4}
\def\yDfive{3.4}
\def\yDsix{2.4}
\def\yDseven{1.4}
\def\halfheight{0.22}

\foreach \x in {0,10,20,30,40,50,60}{
    \draw[gridline] (\x,\yaxis) -- (\x,4.1);
    \draw (\x,\yaxis) -- (\x,0.18);
    \node[below=2pt] at (\x,0.18) {$\x$};
}
\draw[->, thick] (-1,\yaxis) -- (62.5,\yaxis);

\node[anchor=east, font=\small] at (-3,\yDfive) {$D_5$};
\node[anchor=east, font=\small] at (-3,\yDsix) {$D_6$};
\node[anchor=east, font=\small] at (-3,\yDseven) {$D_7$};

\draw[block] (2,\yDfive-\halfheight) rectangle (10,\yDfive+\halfheight);

\draw[block] (4,\yDsix-\halfheight) rectangle (12,\yDsix+\halfheight);
\draw[block] (16,\yDsix-\halfheight) rectangle (24,\yDsix+\halfheight);

\draw[block] (6,\yDseven-\halfheight) rectangle (14,\yDseven+\halfheight);
\draw[block] (18,\yDseven-\halfheight) rectangle (26,\yDseven+\halfheight);
\draw[block] (34,\yDseven-\halfheight) rectangle (42,\yDseven+\halfheight);
\draw[block] (46,\yDseven-\halfheight) rectangle (54,\yDseven+\halfheight);

\node[anchor=west, font=\scriptsize] at (57,\yDfive) {1 component};
\node[anchor=west, font=\scriptsize] at (57,\yDsix) {2 components};
\node[anchor=west, font=\scriptsize] at (57,\yDseven) {4 components};

\draw[dashed, gray!70, thick] (14,\yDsix-0.55) -- (14,\yDfive+0.55);
\node[above, color=gray!70!black, font=\scriptsize] at (14,4.08) {$E_6=14$};

\draw[dashed, gray!70, thick] (30,\yDseven-0.55) -- (30,\yDsix+0.55);
\node[above, color=gray!70!black, font=\scriptsize] at (30,3.08) {$E_7=30$};

\end{tikzpicture}
\caption{Component doubling in the family $V_n$
(Theorem~\ref{thm:Vn}). Each horizontal block represents a maximal run
of even integers spaced by $2$. The recursion
$D_n=P_{E_n}(D_{n-1})$, with $E_n=2^{n-2}-2$, sends $D_{n-1}$
to two separated reflected copies $E_n-D_{n-1}$ and $E_n+D_{n-1}$.
Thus the number of components doubles at each step.}
\label{fig:Vn-doubling}
\end{figure}

\begin{lemma}\label{lem:Vn-ad}
For $n \geq 5$, $V_n \in \G_n$ with right anti-diagonal $e_i = 2^{n-i-1}
- 2$ for $1 \leq i \leq n - 2$ and $e_{n-1} = 1$.
\end{lemma}

\begin{proof}
The first $n - 1$ terms of $V_n$ form the doubling sequence $U_{n-1}$,
whose triangle is the powers-of-two array. The final gap is
$v_n - v_{n-1} = (2^{n-1} - 1) - (2^{n-2} + 1) = 2^{n-2} - 2$,
giving $e_1 = 2^{n-2} - 2$. The rightmost entry of $U_{n-1}$'s row $1$
is $2^{n-3}$, so $e_2 = |2^{n-2} - 2 - 2^{n-3}| = 2^{n-3} - 2$.
Iterating, $e_i = 2^{n-i-1} - 2$ for $1 \leq i \leq n - 2$; in
particular $e_{n-2} = 0$. Finally $e_{n-1} = |1 - 0| = 1$.
\end{proof}

Throughout this section write $D_n := D_{V_n}$ for the distance set of
$V_n$, and recall the preimage map $P_e$ of Section~\ref{sec:reverse}.
By Lemma~\ref{lem:Vn-ad} the right anti-diagonal of $V_n$ is
$(e_1, \ldots, e_{n-1}) = (2^{n-2} - 2,\, 2^{n-3} - 2,\, \ldots,\, 2,\,
0,\, 1)$, and the anti-diagonal of $V_{n}$ is precisely that of
$V_{n-1}$ with one new leading entry $E_n := 2^{n-2} - 2$ prepended.
Consequently the reverse-tree process (Proposition~\ref{prop:reverse})
for $V_n$ is the process for $V_{n-1}$ followed by one additional
preimage step:
\begin{equation}\label{eq:Dn-recursion}
D_n = P_{E_n}(D_{n-1}), \qquad E_n = 2^{n-2} - 2 \quad (n \geq 6).
\end{equation}

We first isolate the arithmetic of the extremes of $D_n$, since the
component count and cardinality both depend on it.

\begin{lemma}[Extremes and structure of $D_n$]\label{lem:Dn-extremes}
For every $n \geq 5$ the set $D_n$ consists of positive even integers,
and
\[
\min D_n = 2n - 8, \qquad \max D_n = 2^{n-1} - 2n + 4 .
\]
Moreover, for $n \geq 6$ the recursion \eqref{eq:Dn-recursion} acts as
two disjoint reflected copies:
\begin{equation}\label{eq:Dn-split}
D_n = (E_n - D_{n-1}) \,\sqcup\, (E_n + D_{n-1}),
\end{equation}
where every element of $E_n - D_{n-1}$ is strictly smaller than every
element of $E_n + D_{n-1}$.
\end{lemma}

\begin{proof}
We argue by induction on $n$.

\emph{Base case $n = 5$.} The reverse-tree computation in
Example~\ref{ex:V5-V6} gives $D_5 = \{2, 4, 6, 8, 10\}$, all positive
and even, with $\min D_5 = 2 = 2(5) - 8$ and
$\max D_5 = 10 = 2^{4} - 2(5) + 4$. This establishes the base case.
(The split \eqref{eq:Dn-split} is asserted only for $n \geq 6$.)

\emph{Inductive step.} Fix $n \geq 6$ and assume the statement for
$n - 1$; in particular $D_{n-1}$ consists of positive even integers with
\begin{equation}\label{eq:IH}
\min D_{n-1} = 2n - 10, \qquad
\max D_{n-1} = 2^{n-2} - 2n + 6 .
\end{equation}

Recall that
\[
P_{E_n}(D_{n-1}) = \{E_n + d : d \in D_{n-1}\}
\;\cup\; \{E_n - d : d \in D_{n-1},\ E_n \geq d\}.
\]
We first show the second branch is unconditional, i.e.\ that
$E_n \geq d$ for every $d \in D_{n-1}$. It suffices to check
$E_n \geq \max D_{n-1}$. Using \eqref{eq:IH},
\begin{equation}\label{eq:gap-positive}
E_n - \max D_{n-1}
= (2^{n-2} - 2) - (2^{n-2} - 2n + 6)
= 2n - 8 \;\geq\; 4 \;>\; 0,
\end{equation}
since $n \geq 6$. Hence $E_n > \max D_{n-1} \geq d$ for all
$d \in D_{n-1}$, so both branches are active and
$D_n = (E_n - D_{n-1}) \cup (E_n + D_{n-1})$.

\emph{Disjoint and ordered.} Every element of the lower branch
satisfies $E_n - d < E_n$ (as $d > 0$), while every element of the
upper branch satisfies $E_n + d > E_n$. Hence each element of
$E_n - D_{n-1}$ is strictly below $E_n$ and each element of
$E_n + D_{n-1}$ strictly above, giving \eqref{eq:Dn-split} with the
claimed ordering; in particular the union is disjoint.

\emph{Parity and positivity.} Each $d \in D_{n-1}$ is even and $E_n =
2^{n-2} - 2$ is even, so $E_n \pm d$ is even. The smallest element of
$D_n$ is $E_n - \max D_{n-1} = 2n - 8 > 0$ by \eqref{eq:gap-positive};
hence all elements of $D_n$ are positive.

\emph{New extremes.} By the ordering in \eqref{eq:Dn-split},
\[
\min D_n = E_n - \max D_{n-1} = 2n - 8,
\]
\[
\max D_n = E_n + \max D_{n-1}
= (2^{n-2} - 2) + (2^{n-2} - 2n + 6) = 2^{n-1} - 2n + 4 .
\]
These are the claimed formulas at $n$, completing the induction.
\end{proof}

With the structure of $D_n$ in hand, the main theorem follows.

\begin{theorem}[Exponentially many components]\label{thm:Vn}
For every $n \geq 5$,
\[
|K_{V_n}| = 5 \cdot 2^{n-4},
\quad
\#\textup{comp}(K_{V_n}) = 2^{n-4},
\quad
h(V_n) = 3 \cdot 2^{n-4} - 2n + 5.
\]
\end{theorem}

\begin{proof}
We first show, by induction on $n \geq 5$, that
\begin{equation}\label{eq:Dn-card-comp}
|D_n| = 5 \cdot 2^{n-5}
\qquad\text{and}\qquad
\#\textup{comp}(D_n) = 2^{n-5}.
\end{equation}
For $n = 5$, $D_5 = \{2,4,6,8,10\}$ has $|D_5| = 5 = 5 \cdot 2^0$ and is
a single run, so $\#\textup{comp}(D_5) = 1 = 2^0$. For $n \geq 6$,
Lemma~\ref{lem:Dn-extremes} gives the disjoint union
\eqref{eq:Dn-split}; since $d \mapsto E_n + d$ and $d \mapsto E_n - d$
are injective, each copy has $|D_{n-1}|$ elements, so $|D_n| = 2
|D_{n-1}| = 5 \cdot 2^{n-5}$.

For the component count, the gap separating the two copies in
\eqref{eq:Dn-split} is
\[
\big(\min(E_n + D_{n-1})\big) - \big(\max(E_n - D_{n-1})\big)
= (E_n + \min D_{n-1}) - (E_n - \min D_{n-1})
= 2 \min D_{n-1} \geq 4
\]
by Lemma~\ref{lem:Dn-extremes} ($\min D_{n-1} = 2n - 10 \geq 2$ for
$n \geq 6$). A gap of at least $4$ between consecutive even integers
breaks the run, so the two copies lie in distinct parity-lattice
components, and within each copy the reflection $d \mapsto E_n \pm d$
preserves adjacency of common difference $2$. Hence
$\#\textup{comp}(D_n) = 2\,\#\textup{comp}(D_{n-1}) = 2^{n-5}$, proving
\eqref{eq:Dn-card-comp}.

By Lemma~\ref{lem:Dn-extremes}, $\min D_n = 2n - 8 > 0$, so $0 \notin
D_n$ and the reflection of Theorem~\ref{thm:symmetry} produces two
disjoint translated copies
$K_{V_n} = (v_n - D_n) \,\sqcup\, (v_n + D_n)$, separated by a gap of
$2\min D_n \geq 4$. Hence
\[
|K_{V_n}| = 2|D_n| = 5 \cdot 2^{n-4},
\qquad
\#\textup{comp}(K_{V_n}) = 2\,\#\textup{comp}(D_n) = 2^{n-4}.
\]

Finally, by Lemma~\ref{lem:Vn-ad},
\[
A(V_n) = \sum_{i=1}^{n-2}\!\big(2^{n-i-1} - 2\big) + 1
= (2^{n-1} - 2) - 2(n - 2) + 1 = 2^{n-1} - 2n + 3,
\]
so by Lemma~\ref{lem:CS-count} the candidate set has size
$A(V_n) + 2 = 2^{n-1} - 2n + 5$, and
\[
h(V_n) = |C_{V_n}| - |K_{V_n}|
= (2^{n-1} - 2n + 5) - 5 \cdot 2^{n-4}
= 3 \cdot 2^{n-4} - 2n + 5. \qedhere
\]
\end{proof}

\begin{corollary}\label{cor:Vn-lower}
For every $n \geq 5$,
\[
\max_{S \in \G_n} \#\textup{comp}(K_S) \geq 2^{n-4}
\quad \text{and} \quad
\max_{S \in \G_n} h(S) \geq 3 \cdot 2^{n-4} - 2n + 5.
\]
\end{corollary}

The first inequality is sharp for $n \leq 10$ (by exhaustive
computation); the second is not, in general.

\section{Computational data}\label{sec:data}

Table~\ref{tab:data} records enumeration data for \(\G_n\) computed using
the corrected extension test, as well as
extremal-width data for $2 \leq n \leq 10$, with $N_{11}$ included
since it is computable from the $\G_{10}$ frontier. The data was
generated using the reverse-tree algorithm of
Proposition~\ref{prop:reverse}; the code in Section~\ref{sec:code}
reproduces the full table.

\begin{table}[htbp]
\centering
\small
\begin{tabular}{c r r c r c r r r r}
\toprule
$n$ & $N_n$ & $m_n$ & \#min & $M_n$ & \#max & \#i.c. & \#defective & max def & max comp \\
\midrule
2  & 1               & 3 & 1 & 3   & 1 & 1         & 0       & 0   & 1 \\
3  & 1               & 5 & 1 & 5   & 1 & 1         & 0       & 0   & 1 \\
4  & 2               & 5 & 1 & 9   & 1 & 2         & 0       & 0   & 1 \\
5  & 6               & 5 & 1 & 17  & 1 & 5         & 1       & 1   & 2 \\
6  & 27              & 5 & 1 & 33  & 1 & 22        & 5       & 5   & 4 \\
7  & 180             & 5 & 1 & 65  & 1 & 120       & 60      & 15  & 8 \\
8  & 1{,}786           & 5 & 1 & 129 & 1 & 1{,}026     & 760     & 47  & 16 \\
9  & 26{,}094          & 5 & 1 & 257 & 1 & 12{,}782    & 13{,}312  & 121 & 32 \\
10 & 559{,}127         & 5 & 1 & 513 & 1 & 237{,}073   & 322{,}054 & 281 & 64 \\
11 & 17{,}535{,}396 & -- & -- & -- & -- & --        & --      & --  & -- \\
\bottomrule
\end{tabular}
\caption{Enumeration data for $\G_n$ computed using the corrected
extension test, together with extremal extension-width data. Here
``\#i.c.'' is the count of interval-complete sequences ($h(S)=0$),
``\#defective'' is the count with $h(S)>0$, and components are counted
in the parity lattice. For $n=11$ only $N_{11}$ is recorded;
per-sequence statistics for $n=11$ were not enumerated.}
\label{tab:data}
\end{table}

For $2 \leq n \leq 10$ the unique maximizer of $|K_S|$ is the doubling
sequence $U_n$, and the unique minimizer is $L_n$. The
fraction of defective sequences grows from $1/6 \approx 17\%$ at
$n = 5$ to $322054/559127 \approx 57.6\%$ at $n = 10$.

\subsection{OEIS connections}\label{sec:oeis}

The enumeration $N_n = |\G_n|$ coincides, after an index shift, with
OEIS~\cite{A080839} (``number of positive increasing integer sequences
of length $n$ with Gilbreath transform $(1,1,1,\ldots)$''), whose terms
are $1, 1, 1, 2, 6, 27, 180, 1786, 26094, 559127, 17535396, \ldots$.
This provides an independent confirmation of our corrected values
$N_2, \ldots, N_{11} = 1, 1, 2, 6, 27, 180, 1786, 26094, 559127,
17535396$, and we attribute the enumeration to that entry rather than
claiming it as new. A comment of T.~D.~Noe on~\cite{A080839} further
records that the extremal (slowest- and fastest-growing) length-$n$
sequences are the minimal sequence and the doubling sequence,
consistent with our Theorems~\ref{thm:min} and~\ref{thm:doubling}.

By contrast, we did not find OEIS entries matching the
interval-complete counts
\[
1, 1, 2, 5, 22, 120, 1026, 12782, 237073,
\]
the maximum-defect sequence
\[
0, 0, 0, 1, 5, 15, 47, 121, 281,
\]
or the $V_n$ extension-set width $5 \cdot 2^{n-4}$; to the best of our
knowledge the structural theory of $K_S$ developed here
(interval-completeness criterion, holes, defect, and the $V_n$ family)
is new. For broader context on the iterated-difference and difference-triangle
literature, related OEIS entries include A036262~\cite{A036262}
(the Gilbreath array of the primes), A036261~\cite{A036261}
(the corresponding iterated absolute differences), A054977~\cite{A054977}
(the conjectured leftmost column), A173816~\cite{A173816} (row sums),
and A347924--A347925~\cite{A347924,A347925} (Gatti polynomial coefficient
numerators and denominators). None of these coincide with the interval-complete,
defect, or component sequences above. As exhaustive sequence search is
delicate, we would welcome verification of these originality claims.

\section{Discussion}\label{sec:discussion}

The framework of this paper has three complementary readings.

From the additive-combinatorics side, the signed-sum set associated to
a finite Gilbreath sequence is a subset-sum set
(Theorem~\ref{thm:subset-reformulation}), and the coincidence
\(S_{\pm}=C_S\) is governed by the classical Brown completeness criterion
applied to the weight multiset \(W_S=\{e_1,\ldots,e_{n-1},1\}\). The
interval-completeness criterion of Theorem~\ref{thm:criterion} is the
ordered analogue of Brown's criterion: the same ``next weight is at
most one plus the sum of previous weights'' shape, but read in the
fixed anti-diagonal order forced by the folding recurrence.

From the dynamical-systems side, the valid distance set is the fiber
over the apex value \(1\) of an ordered composition of folding maps
\(x\mapsto |x-e_i|\). The reverse-tree algorithm of
Section~\ref{sec:reverse} solves the corresponding inverse problem
explicitly. This places the work alongside the Proth--Gilbreath
operator analysis of Bhat, Cobeli, and Zaharescu~\cite{bcz2024}, which
studies the forward dynamics of the same triangle. Our results
contribute the structural analysis of the inverse direction for finite
prefixes.

From the probabilistic side, the conjectures recorded in
Section~\ref{sec:open} ask for the asymptotic distribution of the
defect and component count over a uniformly random sequence
\(S\in\G_n\). These are the finite, deterministic counterparts of the
small-gap probabilistic questions resolved by Chase~\cite{chase}, who
studies whether infinite sequences with random small gaps are
Gilbreath. The framework of this paper makes such finite-distribution
questions concrete: each one is a statement about the distribution of
\(F_S^{-1}(\{1\})\) as \(S\) ranges over \(\G_n\).

In all three readings the central object is the same. The
interval-completeness theorem is simultaneously a sharp completeness
criterion for an ordered subset-sum problem, a structure theorem for
finite fibers of folding-map compositions, and a deterministic
companion to the random Gilbreath models studied recently.

\section{Open questions}\label{sec:open}

\begin{enumerate}[label=\textup{(\arabic*)}]
\item (Conjecture~\ref{conj:max}) Is $M_n = 2^{n-1} + 1$ for all $n$,
uniquely achieved by $U_n$?
\item Asymptotics of \(p_n:=\#\{S\in\G_n:h(S)=0\}/N_n\). Data through
\(n=10\) shows
\[
p_n=1,\,1,\,1,\,\tfrac{5}{6},\,\tfrac{22}{27},\,\tfrac{120}{180},\,\tfrac{1026}{1786},\,\tfrac{12782}{26094},\,\tfrac{237073}{559127}
\;\approx\;0.424.
\]
Does \(p_n\) tend to a limit? More generally, what are the asymptotic
distributions of \(h(S)\), \(|K_S|\), and \(\#\textup{comp}(K_S)\)
under uniform sampling on \(\G_n\)? These are finite, deterministic
analogues of the probabilistic Gilbreath questions resolved by
Chase~\cite{chase}.
\item Closed form for $\max_{S \in \G_n} h(S)$. The lower bound
$3 \cdot 2^{n-4} - 2n + 5$ from Corollary~\ref{cor:Vn-lower} is not
tight.
\item Is $\max_{S \in \G_n} \#\textup{comp}(K_S) = 2^{n-4}$ for all
$n \geq 5$? Verified for $n \leq 10$.
\item Stability classification near the minimum: characterize
$S \in \G_n$ with $|K_S| \leq 9$.
\end{enumerate}

\section{Reproducible code}\label{sec:code}

The following Python module computes the data in Table~\ref{tab:data}
using only the right anti-diagonal state rather than repeatedly storing
and rebuilding full difference triangles. This makes the computation
substantially faster than a direct triangle-based reference
implementation. The program also verifies the first-hole example and
computes \(N_{11}\) by summing the number of increasing valid extensions
from the length-\(10\) frontier.

\begin{lstlisting}
from functools import lru_cache

def preimage_step(e, T):
    """
    Preimages of T under x -> |x-e|, with x >= 0.
    For each t in T, the solutions are x=e+t and, if e>=t, x=e-t.
    """
    out = set()
    for t in T:
        out.add(e + t)
        if e >= t:
            out.add(e - t)
    return tuple(sorted(out))

@lru_cache(maxsize=None)
def valid_distances_from_antidiagonal(e_tuple):
    """
    Given the right anti-diagonal (e_1,...,e_{n-1}), return
    D_S = {|k-s_n| : k in K_S}.
    """
    T = (1,)
    for e in reversed(e_tuple):
        T = preimage_step(e, T)
    return T

def child_antidiagonal(e_tuple, d):
    """
    If d=|k-s_n| is a valid positive distance and k=s_n+d, return the
    right anti-diagonal after appending k.

    Old anti-diagonal: (e_1,...,e_{n-1}).
    New anti-diagonal: (d, |d-e_1|, ||d-e_1|-e_2|, ..., 1).
    """
    r = d
    new_e = [r]
    for e in e_tuple:
        r = abs(r - e)
        new_e.append(r)
    assert new_e[-1] == 1
    return tuple(new_e)

def width_from_distances(D):
    """
    Full extension width |K_S| from the distance set D.
    Distance 0 contributes one extension; each positive distance
    contributes two symmetric extensions.
    """
    return 2 * len(D) - (1 if 0 in D else 0)

def valid_extensions_from_state(sn, e_tuple):
    """
    Full two-sided valid-extension set K_S.
    """
    D = valid_distances_from_antidiagonal(e_tuple)
    out = set()
    for d in D:
        out.add(sn + d)
        out.add(sn - d)
    return tuple(sorted(out))

def candidate_set_from_state(sn, e_tuple):
    """
    Candidate set C_S.
    """
    A = sum(e_tuple)
    return tuple(k for k in range(sn - A - 1, sn + A + 2)
                 if (k - sn) % 2 == 0)

def is_interval_complete(e_tuple):
    """
    Check the criterion e_i <= 1 + sum_{j>i} e_j for all i<=n-2.
    Here e_tuple = (e_1,...,e_{n-1}).
    """
    tail_sum = e_tuple[-1]  # e_{n-1}=1
    for e in reversed(e_tuple[:-1]):
        if e > 1 + tail_sum:
            return False
        tail_sum += e
    return True

def components_count(vals, step=2):
    """
    Number of connected components in one parity lattice.
    """
    vals = sorted(set(vals))
    if not vals:
        return 0
    count = 1
    for a, b in zip(vals, vals[1:]):
        if b - a != step:
            count += 1
    return count

def K_components_count(sn, e_tuple):
    """
    Number of connected components of K_S in the parity lattice.
    """
    return components_count(valid_extensions_from_state(sn, e_tuple), step=2)

def generate_states(max_n):
    """
    Generate states for G_n up to max_n.

    A state is (s_n, e_tuple, seq), where:
      s_n     = last term,
      e_tuple = right anti-diagonal,
      seq     = full sequence, kept only for reporting examples.
    """
    states = [(3, (1,), (2, 3))]
    by_n = {2: states}

    for n in range(3, max_n + 1):
        next_states = []
        for sn, e_tuple, seq in states:
            D = valid_distances_from_antidiagonal(e_tuple)
            for d in D:
                if d > 0:  # increasing extension k=s_n+d
                    k = sn + d
                    new_e = child_antidiagonal(e_tuple, d)
                    next_states.append((k, new_e, seq + (k,)))
        states = next_states
        by_n[n] = states
        print(f"generated G_{n}: {len(states)} sequences")

    return by_n

def summarize_states(states):
    """
    Compute one row of the numerical data table.
    """
    N = len(states)

    min_width = None
    max_width = None
    num_min = 0
    num_max = 0

    num_complete = 0
    num_defective = 0
    max_defect = 0
    max_components = 0

    min_seq = None
    max_seq = None
    max_defect_seq = None
    max_components_seq = None

    for sn, e_tuple, seq in states:
        D = valid_distances_from_antidiagonal(e_tuple)
        width = width_from_distances(D)
        defect = sum(e_tuple) + 2 - width
        comp = K_components_count(sn, e_tuple)

        if min_width is None or width < min_width:
            min_width = width
            num_min = 1
            min_seq = seq
        elif width == min_width:
            num_min += 1

        if max_width is None or width > max_width:
            max_width = width
            num_max = 1
            max_seq = seq
        elif width == max_width:
            num_max += 1

        if is_interval_complete(e_tuple):
            num_complete += 1
        else:
            num_defective += 1

        if defect > max_defect:
            max_defect = defect
            max_defect_seq = seq

        if comp > max_components:
            max_components = comp
            max_components_seq = seq

    return {
        "N": N,
        "min_width": min_width,
        "num_min": num_min,
        "min_seq": min_seq,
        "max_width": max_width,
        "num_max": num_max,
        "max_seq": max_seq,
        "num_complete": num_complete,
        "num_defective": num_defective,
        "max_defect": max_defect,
        "max_defect_seq": max_defect_seq,
        "max_components": max_components,
        "max_components_seq": max_components_seq,
    }

def print_table(by_n):
    """
    Print the table data for n=2,...,10.
    """
    header = (
        "n | N_n | m_n | #min | M_n | #max | "
        "#ic | #def | max def | max comp | max seq"
    )
    print(header)
    print("-" * len(header))

    for n in range(2, 11):
        stats = summarize_states(by_n[n])
        print(
            n,
            stats["N"],
            stats["min_width"],
            stats["num_min"],
            stats["max_width"],
            stats["num_max"],
            stats["num_complete"],
            stats["num_defective"],
            stats["max_defect"],
            stats["max_components"],
            stats["max_seq"],
            sep=" | "
        )

def compute_N_next(states):
    """
    Given states for G_n, compute N_{n+1} by summing the number of
    positive valid distances.
    """
    total = 0
    for sn, e_tuple, seq in states:
        D = valid_distances_from_antidiagonal(e_tuple)
        total += sum(1 for d in D if d > 0)
    return total

def verify_first_hole():
    """
    Verify the first-hole example S=(2,3,5,9,15).
    """
    S = (2, 3, 5, 9, 15)
    sn = 15
    e_tuple = (6, 2, 0, 1)

    C = candidate_set_from_state(sn, e_tuple)
    K = valid_extensions_from_state(sn, e_tuple)
    H = tuple(sorted(set(C) - set(K)))

    print("\nFirst-hole verification")
    print("S =", S)
    print("right anti-diagonal =", e_tuple)
    print("A(S) =", sum(e_tuple))
    print("C_S =", C)
    print("K_S =", K)
    print("H_S =", H)

def V_sequence(n):
    """
    The component-doubling family V_n.
    """
    assert n >= 5
    return (2,) + tuple(2**(i-1) + 1 for i in range(2, n)) + (2**(n-1) - 1,)

def V_antidiagonal(n):
    """
    Right anti-diagonal of V_n:
    e_i = 2^{n-i-1}-2 for 1<=i<=n-2, and e_{n-1}=1.
    """
    assert n >= 5
    return tuple(2**(n-i-1) - 2 for i in range(1, n-1)) + (1,)

def verify_V_family(up_to=10):
    """
    Verify the V_n formulas for n=5,...,up_to.
    """
    print("\nV_n family verification")
    print("n | V_n | |K| | components | defect")
    for n in range(5, up_to + 1):
        S = V_sequence(n)
        sn = S[-1]
        e_tuple = V_antidiagonal(n)
        D = valid_distances_from_antidiagonal(e_tuple)
        width = width_from_distances(D)
        comp = K_components_count(sn, e_tuple)
        defect = sum(e_tuple) + 2 - width
        print(n, S, width, comp, defect, sep=" | ")

if __name__ == "__main__":
    by_n = generate_states(10)
    print()
    print_table(by_n)

    N11 = compute_N_next(by_n[10])
    print("\nN_11 =", N11)

    verify_first_hole()
    verify_V_family(10)
\end{lstlisting}

On a standard laptop, this anti-diagonal-state implementation produces
the table through \(n=10\) and computes \(N_{11}\) in well under a minute.
Runtime will vary by machine.

\section*{Acknowledgments}

This project made use of AI tools during the exploratory stage,
including computational experimentation, conjecture generation, and
preliminary drafting. The computational claims were verified against an
independent implementation; responsibility for the mathematical
statements and proofs rests with the author, and the arguments are
offered for expert review.

\end{document}